\documentclass[11pt]{amsart}
\usepackage[T1]{fontenc}
\usepackage[utf8]{inputenc}
\usepackage{amsmath,amsthm,amsfonts,amssymb,geometry,mathrsfs} 
\usepackage{mathtools}
\usepackage{caption,graphicx,subfig,float,comment}
\usepackage{tikz-cd}
\definecolor{darkgreen}{rgb}{0,0.5,0}
\usepackage[
    colorlinks, 
    citecolor=darkgreen,
    colorlinks=true,
]{hyperref}
\usepackage{cleveref}
\usepackage{multirow, tabularx}
\usepackage{empheq}
\usepackage{empheq}
\usepackage{eqparbox}

\usepackage[
    backend=biber,
    style=alphabetic,
    minalphanames=3,
    maxnames=99,
    maxalphanames=4,
    giveninits=true,
    isbn=false,
    doi=false,
]{biblatex}
\newtheorem{theorem}{Theorem}[section]
\newtheorem{lemma}[theorem]{Lemma}
\newtheorem{proposition}[theorem]{Proposition}

\newtheorem{conjecture}[theorem]{Conjecture}

\theoremstyle{remark}

\theoremstyle{definition}
\newtheorem{definition}[theorem]{Definition}
\newtheorem{example}[theorem]{Example}

\AddToHook{env/definition/begin}{\crefalias{theorem}{definition}}
\AddToHook{env/proposition/begin}{\crefalias{theorem}{proposition}}
\AddToHook{env/lemma/begin}{\crefalias{theorem}{lemma}}
\AddToHook{env/conjecture/begin}{\crefalias{theorem}{conjecture}}
\AddToHook{env/corollary/begin}{\crefalias{theorem}{corollary}}
\AddToHook{env/remark/begin}{\crefalias{theorem}{remark}}
\AddToHook{env/example/begin}{\crefalias{theorem}{example}}
\AddToHook{env/condition/begin}{\crefalias{theorem}{condition}}
\AddToHook{env/notation/begin}{\crefalias{theorem}{notation}}

\makeatletter
\def\qed@tag@alignat{%
    \global\tag@true \nonumber
    &\omit\setboxz@h {\strut@ \qedsymbol}%
    \iftagsleft@%
        \global\advance\tagshift@-\displaywidth%
    \fi%
    \tagsleft@false
    \place@tag
    \kern-\tabskip
    \ifst@rred \else \global\@eqnswtrue \fi \global\advance\row@\@ne \cr
  }

\def\alignat@qed{%
    \ifmeasuring@ \tag*{\qedsymbol}%
    \else \let\math@cr@@@\qed@tag@alignat
    \fi
  }
  \@xp\let\csname alignat*@qed\endcsname\alignat@qed
\makeatother

\newcommand{\bA}{\mathbb A}

\newcommand{\rd}{\mathrm d}

\newcommand{\bG}{\mathbb G}

\newcommand{\fl}{\mathfrak l}

\newcommand{\MT}{\mathrm{MT}}
\newcommand{\bP}{\mathbb P}

\newcommand{\fp}{\mathfrak p}

\newcommand{\bQ}{\mathbb Q}
\newcommand{\bZ}{\mathbb Z}
\newcommand{\cO}{\mathcal{O}}

\newcommand{\fu}{\mathfrak{u}}
\newcommand{\rZ}{\mathrm{Z}}

\newcommand{\ur}{{\mathrm{ur}}}

\newcommand{\mot}{{\mathrm{mot}}}

\newcommand{\Gon}{{\mathrm{Gon}}}
\newcommand{\ev}{{\mathrm{ev}}}

\newcommand{\eps}{\varepsilon}

\newcommand{\Gal}{\mathrm{Gal}}

\DeclareMathOperator{\ab}{ab}

\DeclareMathOperator{\ad}{ad}

\DeclareMathOperator{\gr}{gr}

\DeclareMathOperator{\dR}{dR}

\DeclareMathOperator{\loc}{loc}
\DeclareMathOperator{\Li}{Li}
\DeclareMathOperator{\Lie}{Lie}

\DeclareMathOperator{\Hom}{Hom}

\DeclareMathOperator{\PL}{PL}

\DeclareMathOperator{\Sel}{Sel}
\DeclareMathOperator{\Spec}{Spec}

\begin{document}

\title{Nonabelian Chabauty for the Thrice-punctured Line over Cyclotomic Fields}

\author{Minhyong Kim}
\address{Minhyong Kim, International Centre for Mathematical Sciences, 47 Potterrow,  Edinburgh EH8 9BT;
Korea Institute for Advanced Study, 85 Hoegiro, Dongdaemungu, Seoul, South Korea}
\email{minhyong.kim@icms.ac.uk}

\author{Xiang Li}
\address{Xiang Li,
	School of Mathematics,
	The University of Edinburgh,
    James Clerk Maxwell Building,
    Peter Guthrie Tait Road,
    Edinburgh,
    EH9 3FD, United Kingdom
}
\email{X.Li-198@sms.ed.ac.uk}
    
\author{Martin Lüdtke}
\address{Martin Lüdtke,
	Institut für Mathematik,
	Carl von Ossietzky Universität Oldenburg,
	26111 Oldenburg,
	Germany
}
\email{martin.luedtke@uol.de}

\begin{abstract}
In this paper we study the motivic Chabauty--Kim method, which aims to determine the set of $S$-integral points of $\bP^1\smallsetminus \{0,1,\infty\}$, over cyclotomic fields. We focus on the case $K=\bQ(\zeta_8)$ and $S=\left\{(1-\zeta_8)\right\}$, where we obtain explicit polylogarithmic Kim functions up to depth~$4$ and verify Kim's Conjecture for several primes. We also observe and explain that the Chabauty--Kim locus for the polylogarithmic quotient contains, in addition to the $S$-integral points, certain exceptional points arising from roots of unity in~$\bQ_p$.
\end{abstract}

\maketitle
\tableofcontents
\thispagestyle{empty}

\section{Introduction}
\label{sec:introduction}
For a number field~$K$ and a finite set~$S$ of primes of~$K$, it is a classical theorem of Siegel and Mahler that the set $X(\cO_{K,S})$ of $S$-integral points of the thrice-punctured line $X = \bP^1 \smallsetminus \{0,1,\infty\}$ is finite. In~\cite{kim_2005_the}, Kim gave a new proof of this result in the case $K=\bQ$ by developing a far-reaching non-abelian generalisation of Chabauty's method, now known as \emph{nonabelian Chabauty} or the \emph{Chabauty--Kim method}. For each rational prime~$p$ which splits completely in~$K$ and is not divisible by a prime in~$S$, and for each positive integer~$N$, Chabauty--Kim theory produces a subset $X(\cO_{K}\otimes\bZ_p)_{S,N}$ of $X(\cO_K \otimes \bZ_p) = \prod_{\fp \mid p} X(\cO_{K_{\fp}})$ containing $X(\cO_{K,S})$. These \emph{Chabauty--Kim loci} form a descending sequence
$$ X(\cO_K\otimes \bZ_p) \supseteq X(\cO_K\otimes \bZ_p)_{S,1} \supseteq X(\cO_K\otimes \bZ_p)_{S,2}\supseteq \cdots,$$
so they provide increasingly strong constraints on the location of the $S$-integral points $X(\cO_{K,S})$ inside the $p$-adic points. 
Kim's Conjecture is the statement that $X(\cO_K\otimes \bZ_p)_{S,N}=X(\cO_{K,S})$ for sufficiently large $N$. In the case of hyperbolic curves over~$\bQ$, a conjecture of this form was first formulated in~\cite{BDCKW}.

Work by various authors has aimed to make the Chabauty--Kim method effective, i.e., to make the Chabauty--Kim loci $X(\cO_K\otimes \bZ_p)_{S,N}$ computable in practice \cite{ishaidancohen_2016_mixed, ishaidancohen_2020_mixed, DCW:explicitCK, CDC:polylog1, CDC:polylog2,refined-selmer-equations,luedtke_2025_refined}. While prior work has mostly been done in the setting $K = \bQ$, the recent preprint \cite{LL:PolylogNF} develops the necessary foundations for a systematic study of Chabauty--Kim loci for the thrice-punctured line over general number fields and carries this out in examples over quadratic fields. In this paper we go one step further and study the method over cyclotomic fields, focusing on the concrete example $K = \bQ(\zeta_8)$, $S = \{(1-\zeta_8)\}$.

The Chabauty--Kim loci $X(\cO_{K}\otimes\bZ_p)_{S,N}$ are cut out inside $\prod_{\fp\mid p} X(\cO_{\fp})$ by iterated Coleman integrals on $\bP^1 \smallsetminus \{0,1,\infty\}$ in the variables $(z_{\fp})_{\fp \mid p}$. The \emph{polylogarithmic Chabauty--Kim locus} $X(\cO_{K}\otimes\bZ_p)_{S,\PL,N} \supseteq X(\cO_{K}\otimes\bZ_p)_{S,N}$ is more amenable to computation and is cut out by functions only involving $p$-adic (poly)logarithms
\[ \log(z) = \int_{\vec{1}_0}^z \frac{\rd t}{t}, \quad \Li_n(z) \coloneqq \int_{\vec{1}_0}^z \underbrace{\frac{{\rd}t}{t} \cdots \frac{{\rd}t}{t} \frac{{\rd}t}{1-t}}_n   \quad (1 \leq n \leq N)\]
of the $z_{\fp}$, rather than more general \emph{multiple} polylogarithms. 

Our main result is the derivation of explicit equations for the polylogarithmic Chabauty--Kim locus of depth~$N\leq 4$ in the case $K = \bQ(\zeta_8)$, $S = \{(1-\zeta_8)\}$. To state the theorem, we introduce the modified polylogarithm function
\begin{equation}
\label{eq:Ln}
    L_n(z)=\sum_{k=0}^{n-1}\frac{B_k}{k!}\log(z)^k\Li_{n-k}(z),
\end{equation}
where $B_k$ denotes the $k$-th Bernoulli number with the convention $B_1=-1/2$.

\begin{theorem}
\label{mainthm:equations}
Let $K=\bQ(\zeta_8)$, $S=\left\{(1-\zeta_8)\right\}$, and let $p$ be a rational prime which splits completely in $K$. Let \(\sigma_+\) denote complex conjugation and let \(\sigma_i\) denote the automorphism \(\zeta_8\mapsto -\zeta_8\). Label the primes of~$K$ above~$p$ as $\fp_1,\ldots,\fp_4$ with $\fp_2=\sigma_+^*\fp_1$, $\fp_3=\sigma_i^*\fp_1$, and $\fp_4=(\sigma_+\sigma_i)^*\fp_1$.
For $(z_1,z_2,z_3,z_4)\in X(\cO_K\otimes\bZ_p) = \prod_{i=1}^4 X(\cO_{\fp_i})$ write $X_i=\log(z_i)$, $Y_i=L_1(z_i)$, $Z_i=L_2(z_i)$, $W_i=L_3(z_i)$, $V_i=L_4(z_i)$ for $i=1,2,3,4$.
Assume that $L_2(1+ \zeta_8 + \zeta_8^{-1}) \neq 0$ and $L_n(\zeta_8) \pm L_n(-\zeta_8) \neq 0$ for $n=2,3,4$. Then, for each $1\leq N\leq 4$, the depth-$N$ polylogarithmic Chabauty--Kim locus
$X(\cO_K\otimes \bZ_p)_{S,\PL,N}$ is the common solution set of the first $N$ systems of equations below:
\begin{equation}
\tag{E1}\label{eq:intro-E1}
X_1=X_2,\quad X_3=X_4,\quad Y_1=Y_2,\quad Y_3=Y_4;
\end{equation}

\begin{empheq}[left=\empheqlbrace]{equation}
\tag{E2}\label{eq:intro-E2}
\begin{gathered}
    Z_1+Z_2=-Z_3-Z_4,\\
    Z_1+Z_2=c_2\left(-X_3Y_1+X_1Y_3\right).
\end{gathered}
\end{empheq}

\begin{empheq}[left=\empheqlbrace]{equation}
\tag{E3}\label{eq:intro-E3}
\begin{aligned}
    W_1-W_2+W_3-W_4 &= c_{3,1}(X_1+X_3)(Z_1+Z_3) +c_{3,2}(X_1-X_3)(Z_1+Z_4),\\ W_1-W_2-W_3+W_4 &= c_{3,3}(X_1-X_3)(Z_1+Z_3) +c_{3,4}(X_1+X_3)(Z_1+Z_4);
\end{aligned}
\end{empheq}

\begin{empheq}[left=\empheqlbrace]{equation}
\tag{E4}\label{eq:intro-E4}
\begin{aligned}
    V_1+V_2+V_3+V_4 &= c_{4,1}(X_1+X_3)(W_1+W_2+W_3+W_4)\\ &\quad +c_{4,2}(X_1-X_3)(W_1+W_2-W_3-W_4)\\ &\quad +c_{4,3}(X_1-X_3)(X_1+X_3)(Z_1+Z_2),\\ V_1+V_2-V_3-V_4 &= c_{4,4}(X_1-X_3)(W_1+W_2+W_3+W_4)\\ &\quad+c_{4,5}(X_1+X_3)(W_1+W_2-W_3-W_4)\\ &\quad +c_{4,6}(X_1-X_3)^2(Z_1+Z_2)\\ &\quad +c_{4,7}(X_1+X_3)^2(Z_1+Z_2).
\end{aligned}
\end{empheq}
The $p$-adic coefficients $c_2$, $c_{3,1},\ldots,c_{3,4}$ and $c_{4,1},\ldots,c_{4,7}$ can be computed by plugging known $S$-integral points into the equations and using linear algebra, see \Cref{thm:depth2-equations,thm:depth3-equations,thm:depth4-equations}. 
The non-vanishing assumptions are implied by the $p$-adic period conjecture \cite[Conjecture~2.2.11]{ishaidancohen_2020_mixed} and checked numerically for $p < 200$. 
\end{theorem}

Solving the $p$-adic 4-variable systems of equations of \Cref{mainthm:equations}, we computed the loci $X(\cO_K\otimes \bZ_p)_{S,\PL,N}$ for all $N \leq 4$ and primes $p<200$ which split completely in $\bQ(\zeta_8)$; the cardinalities are listed in \Cref{tab:overallstatistics}. The computations show that the loci always contain exceptional points in addition to the $\#X(\cO_{K,S})=75$ $S$-integral points. We can explain the presence of all of these exceptional points, see \Cref{thm:depth2-locus-description,thm:depth3-locus-description,thm:depth4-locus-description}. 
Notably, we show that under the $p$-adic period conjecture, the $p$-adic points of the form $(\zeta,\zeta^{-1},\eta,\eta^{-1})$, with $\zeta,\eta \neq 1$ roots of unity in~$\bQ_p$, persist in the polylogarithmic Chabauty--Kim locus even in infinite depth, see \Cref{thm:rootofunityalwayssol}.
Nevertheless, leveraging the natural $S_3$-action on $X=\bP^1\smallsetminus\{0,1,\infty\}$ to define an intermediate \emph{$S_3$-symmetrised locus}
\[ X(\cO_K \otimes \bZ_p)_{S,\PL,N} \supseteq X(\cO_K \otimes \bZ_p)_{S,\PL,N}^{S_3} \supseteq X(\cO_K \otimes \bZ_p)_{S,N}, \]
we manage to eliminate many of the exceptional points and verify Kim's Conjecture in several cases:

\begin{theorem}
\label{mainthm:s3-sym-kim}
Let $K=\bQ(\zeta_8)$, $S=\left\{(1-\zeta_8)\right\}$. For every prime $p<200$ which splits completely in $K$ and for which $\bQ_p$ contains no primitive sixth root of unity, namely $p=17,41,89,113,137$, Kim's Conjecture holds at depth~$4$. 
\end{theorem}

Let us briefly say something about the proof of \Cref{mainthm:equations}. Deriving equations for the polylogarithmic Chabauty--Kim loci $X(\cO_K \otimes \bZ_p)_{S,\PL,N}$ rests on being able to find explicit formulas describing the product of evaluation maps
\begin{equation}
\label{eq:evaluation-map}
    \Hom_{\gr}(L_S^{\MT}, \Lie(\Pi_{\PL,N}^{\omega}))_{\bQ_p} \to \prod_{\fp \mid p} \Lie(\Pi_{\PL,N}^{\dR})_{K_{\fp}}, \qquad 
    c \mapsto (c(\eps_{\fp}))_{\fp \mid p},
\end{equation}
and finding equations defining the scheme-theoretic image. Here, $L_S^{\MT}$ is the Lie algebra of the unipotent radical of the mixed Tate Galois group of $\cO_{K,S}$; $\Pi^{\dR}_{\PL,N}$ denotes the depth-$N$ polylogarithmic de Rham fundamental group of~$X$; $\Pi^{\omega}_{\PL,N}$ is a canonical $\bQ$-form of~$\Pi^{\dR}_{\PL,N}$; $\Hom_{\gr}(L_S^{\MT}, \Lie(\Pi_{\PL,N}^{\omega}))$ denotes the space of weight-graded Lie algebra homomorphisms, and the elements $\eps_{\fp} \in L_S^{\MT}(K_{\fp})$ for $\fp \mid p$ are the logarithms of the $\fp$-adic period points in $U_S^{\MT}(K_{\fp})$ constructed by Chatzistamatiou--Ünver \cite{chatzistamatiou-unver:p-adic_periods}. If $K/\bQ$ is Galois and $S$ is Galois-stable, the Galois group $\Gal(K/\bQ)$ acts on $L_S^{\MT}$ and permutes the various period elements $\eps_{\fp}$ for $\fp \mid p$. This, together with a choice of free generators of $L_S^{\MT}$ which interact nicely with the Galois action, allows us to find the general shape of the polynomials describing the map~\eqref{eq:evaluation-map}. It then remains to determine the $p$-adic coefficients, which can be achieved if one has a sufficiently large supply of $S$-integral points available. This works especially well over cyclotomic fields, where one can systematically produce many solutions to the $S$-unit equation (see §\ref{sec:producing-points}). In the case $K = \bQ(\zeta_8)$, $S = \{(1-\zeta_8)\}$, the number of solutions is $\#X(\cO_{K,S}) = 75$, as one can check with the $S$-unit equation solver in SageMath \cite{sagemath}.

This paper is dedicated to Barry Mazur and Ken Ribet. Their epoch-making papers on the arithmetic geometry of modular curves~\cite{Mazur:Eisenstein,Ribet:unramified-p-extensions} were overwhelming influences in the evolution of ideas leading up to this paper. The active study of finite flat group schemes in both papers were critical to the development of $p$-adic Hodge theory. Of course, $\bP^1 \smallsetminus \left\{0,1,\infty\right\}$ is a modular curve. The determination of integral and rational points on modular curves extending the results of~\cite{Mazur:Eisenstein} have been so far among the most concrete applications of the Chabauty--Kim method~\cite{BalakrishnanMazur:Ogg,BDMTV:QCModularCurves,BDMTV:splitCartan13}.

\section{Motivic Chabauty--Kim over number fields}

We begin by recalling the construction of Chabauty--Kim loci over number fields, referring to \cite{LL:PolylogNF} and further references therein for details. As in the introduction, we fix a number field $K$, a finite set~$S$ of primes of~$K$, and let $X=\bP^1\smallsetminus \left\{0,1,\infty\right\}$ be the thrice-punctured line over the ring of $S$-integers~$\cO_{K,S}$. We also fix a rational prime~$p$ which splits completely in~$K$ and which is not divisible by a prime in~$S$. We use the motivic formulation of the Chabauty--Kim method developed by Hadian~\cite{Hadian2011} and Dan-Cohen and Wewers \cite{DCW:explicitCK, ishaidancohen_2016_mixed}, in which the pro-unipotent étale fundamental group of the thrice-punctured line is replaced by the motivic fundamental group of Deligne and Goncharov \cite{deligne-goncharov}.

\subsection{The motivic Chabauty--Kim diagram}
Denote by $\pi_1^{\dR}(X,b)$ the de Rham fundamental group of $X = \bP^1 \smallsetminus \{0,1,\infty\}$ at   the tangential base point $b = \vec{1}_0$ \cite{deligne:droite-projective}. This is a pro-unipotent group over~$K$ admitting a canonical $\bQ$-form denoted by $\pi_1^{\omega}(X,b)$. The latter is equipped with an action of the \emph{mixed Tate motivic Galois group} $G_S^{\MT}$ of $\cO_{K,S}$, the Tannaka group of the category $\MT(\cO_{K,S},\bQ)$ of $\bQ$-linear mixed Tate motives over~$\cO_{K,S}$ with respect to the \emph{canonical fibre functor}~$\omega$, a $\bQ$-form of the de Rham realisation functor. There is a canonical decomposition $G_S^{\MT}=U_S^{\MT}\rtimes \bG_m$ with $U_S^{\MT}$ pro-unipotent. The $G_S^{\MT}$-action on $\pi_1^{\omega}(X,b)$ comes from the fact that $\pi_1^{\omega}(X,b)$ has a motivic origin: \emph{the motivic fundamental group} $\pi_1^{\mot}(X,b)$ constructed by Deligne and Goncharov \cite{deligne-goncharov}.
Fix a quotient $\pi_1^{\mot}(X,b) \twoheadrightarrow \Pi$ in $\MT(\cO_{K,S},\bQ)$, denote its canonical realisation by $\Pi^{\omega}$ and its de Rham realisation by~$\Pi^{\dR}$. Equivalently, $\Pi^{\omega}$ is a $G_S^{\MT}$-equivariant quotient of $\pi_1^{\omega}(X,b)$. The central object of study is the \emph{(motivic) Chabauty--Kim diagram}:

\begin{equation}
    \label{eq:motivic-ck-diagram}
    \begin{tikzcd}[column sep = huge]
        X(\cO_{K,S}) \rar[hook] \dar["j_S"] & X(\cO_{K}\otimes_{\bZ}\bZ_p)=\prod_{\fp\mid p} X(\cO_{\fp}) \dar["j_p=\prod_{\fp} j_{\fp}^{\dR}"] \\
        \Sel_{S,\Pi}^{\mot}(X)_{\bQ_p} \rar["\loc_p=\prod_{\fp} \ev_{\eta_{\fp}}"] & \prod_{\fp \mid p} \Pi^{\dR}_{K_{\fp}}.
    \end{tikzcd}
\end{equation}
where $\Sel_{S,\Pi}^{\mot}(X) \coloneqq \rZ^1_{\bG_m}(U_S^{\MT},\Pi^{\omega})$ is the space of $\bG_m$-equivariant cocycles from $U_S^{\MT}$ to $\Pi^{\omega}$, called \emph{motivic Selmer scheme}; $j_S$ is the \emph{global Kummer map}, and $j_{\fp}^{\dR}$ is the \emph{local Kummer map} at the prime $\fp$; $\eta_{\fp}\in U_S^{\MT}(K_{\fp})$ is the \emph{$\fp$-adic period point}, the inverse of  the point $\eta_{\fp}^{\ur}$ defined in \cite[Lemma~2.2.5]{chatzistamatiou-unver:p-adic_periods}; and $\loc_p$ is the \emph{localisation map}, the product of the evaluation maps $\ev_{\eta_{\fp}}$ at $\eta_\fp$. The motivic Selmer scheme parametrises $\Pi$-torsors in $\MT(\cO_{K,S},\bQ)$. Since the diagram commutes, a necessary condition for a point of $X(\cO_K \otimes \bZ_p)$ to come from $X(\cO_{K,S})$ is that its image under $j_p$ lies in the image of the localisation map~$\loc_p$. This motivates the following definition.

\begin{definition}
    The \emph{Chabauty--Kim locus} for the quotient $\Pi$ is the set $$X(\cO_K\otimes \bZ_p)_{S,\Pi}\coloneqq j_p^{-1}(\loc_p(\Sel_{S,\Pi}^{\mot}(X)_{\bQ_p}))\subseteq X(\cO_{K}\otimes_{\bZ}\bZ_p),$$
    i.e., the preimage under the local Kummer map of the scheme-theoretic image of the localisation map.
    By construction, $X(\cO_K\otimes \bZ_p)_{S,\Pi}$ is a subset of $X(\cO_K\otimes \bZ_p)$ containing~$X(\cO_{K,S})$. 
    If $\Pi = \pi_1^{\mot}(X,b)_N$ is the $N$-th descending central series quotient, we denote the associated Chabauty--Kim locus by $X(\cO_K \otimes \bZ_p)_{S,N}$.
\end{definition}

\begin{conjecture}[Kim's Conjecture]
\label{conj:kim-full}
    For sufficiently large quotients $\Pi$, we have
    \( X(\cO_K \otimes_{\bZ} \bZ_p)_{S,\Pi} = X(\cO_{K,S}). \)
\end{conjecture}

\subsection{Chabauty--Kim for the polylogarithmic quotient}
We focus on a particular quotient of the fundamental group called the \emph{polylogarithmic quotient}, which was first introduced by Deligne \cite[§16]{deligne:droite-projective}. The inclusion $X\hookrightarrow \bG_m$ induces a homomorphism $\pi_1^{\mathrm{mot}}(X,b)\to \pi_1^{\mot}(\bG_m)$, whose kernel we denote by~$N_1$. The polylogarithmic quotient is defined by $\Pi_{\mathrm{PL}} \coloneqq \pi_1^{\mathrm{mot}}(X,b)/[N_1,N_1]$. Its canonical and de Rham realisation are denoted by $\Pi_{\PL}^{\omega}$ and $\Pi_{\PL}^{\dR}$, respectively. The Lie algebra of $\pi_1^{\mathrm{\omega}}(X,b)$ is the free pro-nilpotent Lie algebra on two generators $e_0,e_1$, and the Lie algebra of $\Pi_{\mathrm{PL}}^{\omega}$ is its quotient by the Lie-monomials of $e_1$-degree $>1$. We also fix $N \in [1,\infty]$, denote by $\Pi_{\PL,N}$ the maximal quotient of $\Pi_{\PL}$ of nilpotency depth~$N$, and by $\Sel_{S,\PL,N}^{\mot}(X)$ the associated Selmer scheme. A basis of $\Lie( \Pi_{\mathrm{PL},N}^{\omega})$ is given by $e_0$ and $e_1$ in degree $-1$ and $\mathrm{ad}(e_0)^{n-1}(e_1)$ in degree $-n$ for $2 \leq n \leq N$. 
What makes the polylogarithmic quotient particularly convenient to work with is the fact that $U_S^{\MT}$ acts trivially on~$\Pi_{\mathrm{PL}}^{\omega}$ \cite[Lemma 4.7]{LL:PolylogNF}, so cocycles are just homomorphisms: $\Sel_{S,\PL,N}^{\mot}(X) \cong \mathrm{Hom}_{\bG_m}(U_S^{\MT},\Pi_{\mathrm{PL},N}^{\omega})$. This also makes it easy to work on the level of Lie algebras. 

Write $L_S^{\MT}=\Lie (U_S^{\MT})$. 
For each prime $\fp\mid p$, let $\varepsilon_{\fp}\coloneqq \log(\eta_{\fp})\in L_S^{\MT}(K_{\fp})$, where $\log: U_S^{\MT}\to L_S^{\MT}$ is the logarithm map from the pro-unipotent group $U_S^{\MT}$ to its pro-nilpotent Lie algebra. Since homomorphisms of unipotent groups are equivalent to homomorphisms between their Lie algebras, and the logarithm is compatible with homomorphisms, we can rewrite the evaluation maps appearing in the bottom row of the Chabauty--Kim diagram as follows:
$$\begin{tikzcd}[ampersand replacement=\&]
Z^1_{\bG_m}(U_S^{\MT},\Pi_{\mathrm{PL},N}^{\omega})_{\bQ_p}\arrow[r,"\ev_{\eta_{\fp}}"]\arrow[d,equal] \& \Pi_{\mathrm{PL},N,K_{\fp}}^{\dR}\arrow[d,equal]\\
\mathrm{Hom}_{\bG_m}(U_S^{\MT},\Pi_{\mathrm{PL},N}^{\omega})_{\bQ_p}\arrow[r,"\ev_{\eta_{\fp}}"] \arrow[d,"\Lie", "\simeq"'] \& \Pi_{\PL,N,K_{\fp}}^{\dR} \arrow[d,"\log", "\simeq"']\\
\mathrm{Hom}_{\mathrm{gr}}(L_S^{\MT},\Lie(\Pi_{\mathrm{PL},N}^{\omega}))_{\bQ_p}\arrow[r,"\ev_{\varepsilon_{\fp}}"] \& \Lie(\Pi_{\PL,N}^{\dR})_{K_{\fp}}.
\end{tikzcd}$$
As a consequence, determining the scheme-theoretic image of the bottom row of the diagram~\eqref{eq:motivic-ck-diagram} is equivalent to determining the scheme-theoretic image of the product of evaluation maps~\eqref{eq:evaluation-map}. A natural set of coordinates on $\Pi_{\PL,N}^{\dR}$ is given by the functions $L_0,L_1,L_2,\ldots,L_N$ defined as the dual basis of $e_0,e_1,[e_0,e_1],\ldots,\ad(e_0)^{N-1}e_1$. They give rise to the modified polylogarithm functions $L_n^{\fp}\colon X(\cO_{\fp}) \to \cO_{\fp}$ defined in~\eqref{eq:Ln}, where we usually drop the superscript $(-)^{\fp}$:

\begin{lemma}[{\cite[Lemma~4.4]{LL:PolylogNF}}]
    \label{Ln-as-pullback}
    Let $L_0,L_1,L_2,\ldots \in \Lie (\Pi_{\PL}^{\dR})^{\vee}$ denote the dual basis of $e_0, e_1, [e_0,e_1],\ldots$ Then, for all $n \geq 1$, the pullback of $L_n$ along the composition
    \[ X(\cO_{\fp}) \xrightarrow{j_{\fp}^{\dR}} \Pi_{\PL}^{\dR}(K_{\fp}) \xrightarrow[\sim]{\log} \Lie (\Pi_{\PL}^{\dR})(K_{\fp}) \]
    equals the map $L_n$ from~\eqref{eq:Ln}, and the pullback of $L_0$ equals the $\fp$-adic logarithm.
\end{lemma}

In other words, after taking logarithms, the local Kummer map from $X(\cO_{\fp})$ into $\Lie(\Pi_{\mathrm{PL},N}^{\dR})$ is given by
$$z\mapsto (\log(z),L_1(z),L_2(z),\ldots,L_N(z))$$
in the basis $e_0, e_1, [e_0,e_1],\ldots$ of $\Lie(\Pi_{\PL,N}^{\dR})$.
Thus, determining equations for the scheme-theoretic image of the Selmer scheme inside $\prod_{\fp \mid p} \Lie(\Pi_{\PL,N}^{\dR})_{K_{\fp}}$ and pulling them back to $\prod_{\fp \mid p} X(\cO_{\fp})$, we obtain equations defining the polylogarithmic Chabauty--Kim locus $X(\cO_K \otimes \bZ_p)_{S,\PL,N}$ which are naturally given by polynomials with $\bQ_p$-coefficients in $\log(z_{\fp})$ and $L_n(z_{\fp})$ for $\fp \mid p$ and $1\leq n\leq N$.

\subsection{Coordinates on the polylogarithmic Selmer scheme}
\label{sec:selmer-scheme-coordinates}

The Lie algebra $L_S^{\MT}$ is free pro-nilpotent on a countable set of generators, and the $\bG_m$-action induces a product grading $L_S^{\MT} = \prod_{n=1}^{\infty} (L_S^{\MT})_{-n}$ by \emph{half-weight}. If $\Sigma_n \subseteq (L_S^{\MT})_{-n}$ is a lift of a basis of $(L_S^{\MT})^{\ab}_{-n}$ under the abelianisation map, then $\Sigma \coloneqq \coprod_{n = 1}^{\infty} \Sigma_n$ is a set of free generators of $L_S^{\MT}$. For the abelianisation of $L_S^{\MT}$, we have a canonical isomorphism \cite[Eq.~(2.3.11)]{deligne-goncharov}
$$(L_S^{\MT})^{\ab}\cong \prod_{n=1}^\infty (K_{2n-1}(\cO_{K,S})\otimes \bQ)^\vee.$$
We denote by $\tau_1,\ldots,\tau_{d_1} \in (L_S^{\MT})_{-1}$ a choice of generators in degree~$-1$ and by $\sigma_{n,1},\ldots,\sigma_{n,d_n} \in (L_S^{\MT})_{-n}$ generators in degree~$-n$ for $n \geq 2$, where $d_n:=\dim_{\bQ} K_{2n-1}(\cO_{K,S})\otimes \bQ$. 
We obtain a set of coordinates on the polylogarithmic Selmer scheme as follows. For any $\xi \in \mathrm{Hom}_{\mathrm{gr}}(L_S^{\MT},\Lie (\Pi_{\mathrm{PL},N}^{\omega}))$, we have
\begin{align}
\label{eq:tauimapto}
    \xi(\tau_i)&=x_i(\xi)e_0+y_i(\xi)e_1 \text{ for }1\leq i\leq d_1,\\
\label{eq:sigmanmapto}
    \xi(\sigma_{n,i})&=z_{n,i}(\xi)\,\mathrm{ad}(e_0)^{n-1}e_1\text{ for }2\leq n\leq N \text{ and }1\leq i\leq d_n,
\end{align}
which defines functions $x_i,y_i,z_{n,i}\colon \Sel_{S,\PL,N}^{\mot}(X) = \mathrm{Hom}_{\mathrm{gr}}(L_S^{\MT},\Lie (\Pi_{\mathrm{PL},N}^{\omega}))\to \bA^1$ and induces an isomorphism \cite[Theorem~6.4]{motivic-selmer-scheme}
\begin{equation}
\label{eq:Selparametrisation}
  \Sel_{S,\PL,N}^{\mot}(X) \cong \Spec \bQ\bigl[(x_i)_{1\leq i\leq d_1}, (y_i)_{1\leq i\leq d_1}, (z_{n,i})_{2\leq n\leq N,1\leq i\leq d_n}\bigr].  
\end{equation}

\subsection{The cocycle evaluation map}
\begin{definition}
    \label{def:goncharov-quotient}
    The \emph{Goncharov quotient} $L_S^{\MT} \twoheadrightarrow (L_S^{\MT})_{\Gon}$ is the quotient by the \emph{Goncharov ideal} 
    \[ I_{\Gon} \coloneqq [(L_S^{\MT})_{\leq -2}, (L_S^{\MT})_{\leq -2}]. \]
    Here, $(L_S^{\MT})_{\leq -2}$ denotes the Lie ideal of elements of degree $\leq -2$ for the grading by half-weight.
\end{definition}

\begin{proposition}[{\cite[Proposition 4.12]{LL:PolylogNF}}]
    \label{lie-algebra-element-expansion}
    Each element $\eps$ of $L_S^{\MT}(R)$ for a $\bQ$-algebra~$R$ can uniquely be written in the form 
    \begin{equation}
        \label{eq:eps-form}
        \eps = \sum_{n=1}^\infty \eps_n + \eps_{\Gon}
    \end{equation}
    where $\eps_{\Gon}$ is contained in the Goncharov ideal (\Cref{def:goncharov-quotient}), 
    \begin{equation}
        \label{eq:eps-1}
        \eps_1 = \sum_{i=1}^{d_1} c_{\tau_i} \tau_i,
    \end{equation}
    and
    \begin{equation}
        \label{eq:eps-n}
        \begin{split}
            \eps_n &= \sum_{ i_1 \leq \ldots \leq i_{n-1} > i_n} c_{\tau_{i_1}\cdots \tau_{i_n}} \ad(\tau_{i_1})\cdots\ad(\tau_{i_{n-1}}) \tau_{i_n} \\ 
              &\qquad + \sum_{m=2}^n\sum_{i=1}^{d_n} \sum_{i_1\leq\ldots\leq i_{n-m}} c_{\tau_{i_1}\cdots\tau_{i_{n-m}} \sigma_{m,i}} \ad(\tau_{i_1})\cdots\ad(\tau_{i_{n-m}}) \sigma_{m,i}
        \end{split}
    \end{equation}

    for $n \geq 2$, with coefficients $c_w \in R$. Here, all the indices $i_j$ run over $\{1,\ldots,d_1\}$.
\end{proposition}

\begin{theorem}[{\cite[Theorem 4.15]{LL:PolylogNF}}]
    \label{cocycle-evaluation-map}
    Let $R$ be a $\bQ$-algebra and let $\eps \in L_S^{\MT}(R)$ be written as in \Cref{lie-algebra-element-expansion} above. Let $\xi \in \Hom_{\gr}(L_S^{\MT},\Lie (\Pi_{\PL,N}^{\omega}))_R$ be a graded homomorphism defined over~$R$ with coordinates $x_i = x_i(\xi)$, $y_i = y_i(\xi)$ and $z_{n,i} = z_{n,i}(\xi)$ defined by Eqs.~(\ref{eq:tauimapto})--(\ref{eq:sigmanmapto}). Then the cocycle evaluation map $\ev_{\eps}$ is given by
    \[ \ev_{\eps}(\xi) = \sum_{i=1}^{d_1} c_{\tau_i} x_i e_0 + \sum_{i=1}^{d_1} c_{\tau_i} y_i e_1 + \sum_{n=2}^{N} k_n \ad(e_0)^{n-1} e_1\]
    with
    \begin{equation}
    \begin{split}
            k_n &= \sum_{i_1 \leq \ldots \leq i_{n-1}} \sum_{i_n} c_{\tau_{i_1}\cdots\tau_{i_n}}' x_{i_1}\cdots x_{i_{n-1}} y_{i_n} \\
            &\qquad + \sum_{m=2}^n \sum_{i=1}^{d_m} \sum_{i_1 \leq \ldots \leq i_{n-m}} c_{\tau_{i_1}\cdots \tau_{i_{n-m}}\sigma_{m,i}} x_{i_1} \cdots x_{i_{n-m}} z_{m,i}
    \end{split}
    \end{equation}

    and
    \[ 
        c_{\tau_{i_1}\cdots\tau_{i_n}}' = \begin{cases}
            c_{\tau_{i_1}\cdots\tau_{i_n}} & \text{if $i_{n-1} > i_n$},\\[2mm]
            -\displaystyle\sum_{\substack{j \in \{i_1,\ldots,i_{n-1}\},\\ j < i_n}} c_{\tau_{i_1}\cdots\hat\tau_{j}\cdots \tau_{i_n} \tau_{j}} & \text{if $i_{n-1} \leq i_n$.}
        \end{cases}
    \]
\end{theorem}

In half-weight~$-1$, we have an identification $(L_S^{\MT})_{-1} \cong \cO(U_S^{\MT})_1^{\vee}$
and an isomorphism 
\[ \log^{\fu}\colon \cO_{K,S}^\times \otimes \bQ \xrightarrow{\sim} \cO(U_S^{\MT})_1 \]
given by the motivic logarithm \cite[Lemma 5.9]{LL:PolylogNF}. Let $\left\{\alpha_i\right\}_i$ be a $\bQ$-basis of $\cO_{K,S}^\times\otimes\bQ$ and suppose that the $\bQ$-basis $\left\{\tau_i\right\}_i$ of $(L_S^{\mathrm{MT}})_{-1}$ is the dual basis of $\left\{\log^{\fu}(\alpha_i)\right\}_i$:
\[ \langle \tau_i, \log^{\fu}(\alpha_j)\rangle = \delta_{ij}. \]

\begin{lemma}[{\cite[Proposition 5.11]{LL:PolylogNF}}]
\label{lem:logalpha}
In the Goncharov representation of the $\fp$-adic period element $\varepsilon_{\fp}$, the coefficients in half-weight~$-1$ are given by $\fp$-adic logarithms: $c_{\tau_i}=\log (\alpha_i)$ for $i=1,\ldots,d_1$.
\end{lemma}

The choice of the $\tau_i$ defines functions $x_i,y_i$ on the Selmer scheme via \eqref{eq:tauimapto}. On the image of the global Kummer map, they are given in terms of the dual basis $\{\alpha_i^{\vee}\}_i$ of $(\cO_{K,S}^{\times} \otimes \bQ)^{\vee}$:
\begin{lemma}
\label{lem:globalKummervaluation}
If $\xi_z\in \mathrm{Hom}_{\mathrm{gr}}(L_S^{\MT},\Lie(\Pi_{\PL,N}^{\omega}))$ denotes the image of $z\in X(\cO_{K,S})$ under the global Kummer map $j_S$, then $$x_i(\xi_z)=\alpha_i^{\vee}(z),\quad y_i(\xi_z)=-\alpha_i^{\vee}(1-z).$$
\end{lemma}
\begin{proof}
Consider the Chabauty--Kim diagram in \eqref{eq:motivic-ck-diagram} and fix $\fp \mid p$. The image of $z\in X(\cO_{K,S})$ under the composition of the top map and the $\fp$-adic local Kummer map is $(\log(z),-\log(1-z))=(\sum_i \alpha_i^{\vee}(z)\log(\alpha_i),\sum_i -\alpha_i^{\vee}(1-z)\log(\alpha_i))$ in depth $N=1$, while the image under the composition of the global Kummer map and the $\fp$-adic localisation map is equal to $(\sum_i  x_i(\xi_z)\log(\alpha_i), \sum_i y_i(\xi_z)\log(\alpha_i))$, where we used \Cref{cocycle-evaluation-map} and \Cref{lem:logalpha}. The result follows from the commutativity of the Chabauty--Kim diagram.
\end{proof}

The Goncharov coefficients of $\eps_{\fp}$ in half-weight $< -1$ are more difficult to understand. However, we have:
\begin{lemma}[{\cite[Lemma~5.17]{LL:PolylogNF}}]
    \label{period-conjecture-implies-nonvanishing}
    Assume the $\fp$-adic period conjecture \cite[Conjecture~2.2.11]{ishaidancohen_2020_mixed}. Then all coefficients $c_w \in K_{\fp}$ in the Goncharov expansion of $\eps_{\fp}$ are nonzero.
\end{lemma}

\subsection{The Galois action on the motivic Selmer scheme}
Let $\sigma\colon K \to K'$ be a homomorphism of number fields and let $S$ be a set of primes of~$K$. Define the set $\sigma_*S$ of primes of~$K'$ by
$\sigma_*S \coloneqq (\sigma^*)^{-1}(S) \coloneqq \{ \fp' \,\vert\, \sigma^*\fp' \in S \}.$ 
By \cite[§2.16]{deligne-goncharov} and $\cO_{K,S}^{\times} \subseteq \cO_{K',\sigma_*S}^{\times}$, extension of scalars along~$\sigma$ induces a fully faithful, exact $\otimes$-functor 
$
    \sigma_*\colon \MT(\cO_{K,S},\bQ) \to \MT(\cO_{K',\sigma_*S},\bQ) 
$
which is compatible with the canonical fibre functors. On Tannaka groups it induces a surjective homomorphism $\sigma^*\colon G_{\sigma_* S}^{\MT} \to G_S^{\MT}$, hence $\bG_m$-equivariant surjective homomorphisms
\begin{equation*}
    \sigma^*\colon U_{\sigma_* S}^{\MT} \to U_S^{\MT} \quad \text{and} \quad 
    \sigma^*\colon L_{\sigma_* S}^{\MT} \to L_S^{\MT}.
\end{equation*}
Let $\pi_1^{\mot}(X,b) \twoheadrightarrow \Pi$ be a quotient of the motivic fundamental group of~$X$ in $\MT(\cO_{K,S},\bQ)$. We identify $\Pi$ with its image $\sigma_* \Pi$ under the above fully faithful functor. If $P$ is a $\Pi$-torsor over some $\bQ$-algebra~$R$ in $\MT(\cO_{K,S},\bQ)$, then $\sigma_* P$ is a $\Pi$-torsor over~$R$ in $\MT(\cO_{K',\sigma_*S},\bQ)$, thus we get an induced morphism of motivic Selmer schemes
\begin{equation}
    \label{eq:sigma-on-selmer-schemes}
    \sigma_*:\rZ^1_{\bG_m}(U_{S}^{\MT}, \Pi^{\omega})\to \rZ^1_{\bG_m}(U_{\sigma_*S}^{\MT}, \Pi^{\omega})
\end{equation}

\begin{lemma}[{\cite[Lemma 3.10]{LL:PolylogNF}}]
\label{thm:GalactionmotivicSel}
    For $\xi\in \rZ^1_{\bG_m}(U_{S}^{\MT}, \Pi^{\omega})$ and $u\in U_{\sigma_*S}^{\MT}$, the map \eqref{eq:sigma-on-selmer-schemes} is given by
    $$(\sigma_* \xi)(u)=\xi(\sigma^*(u)).$$
\end{lemma}
When $U_S^{\MT}$ acts trivially on~$\Pi^{\omega}$, so that the Selmer scheme is the space of graded homomorphisms $L_S^{\MT} \to \Lie(\Pi^{\omega})$, it follows from \Cref{thm:GalactionmotivicSel} that $(\sigma_*\xi)(\eps) = \xi(\sigma^*(\eps))$ for all $\xi \in \Hom_{\gr}(L_S^{\MT},\Lie(\Pi^{\omega}))$ and $\eps \in L_S^{\MT}$.

\begin{theorem}[{\cite[Theorem 5.19]{LL:PolylogNF}}]
\label{thm:Galpermuteseta}
    We have $\eta_{\sigma^*\fp'}=\sigma^* \eta_{\fp'}$ and $\varepsilon_{\sigma^*\fp'}=\sigma^* \varepsilon_{\fp'}$ for all primes $\fp' \not\in \sigma_*S$ of~$K'$.
\end{theorem}
When $K/\bQ$ is Galois then \Cref{thm:Galpermuteseta} says that $\Gal(K/\bQ)$ permutes the period elements $\eta_{\fp}$ and $\eps_{\fp}$ for $\fp \mid p$.

\section{Cyclotomic fields}
\subsection{Dimensions of Selmer schemes over cyclotomic fields}
Let $K$ be a number field, let $S$ be a finite set of primes of $K$, and let $p$ be a rational prime that splits completely in~$K$ and is not divisible by a prime in~$S$.

\begin{lemma}[{\cite{borel}}]
\label{lem:Kdim}
    $K_{2n-1}(\cO_{K,S})\otimes \bQ$ has the $\bQ$-dimension
$$d_n=\begin{cases}
r_1+r_2+\#S-1, & n=1,\\
r_1+r_2, & n\geq 3\quad \mathrm{odd},\\
r_2, & n\geq 2\quad \mathrm{even},
\end{cases}$$
where $r_1,r_2$ denote the number of real embeddings and pairs of complex embeddings of $K$.
For $n=1$, there is a canonical isomorphism $K_1(\cO_{K,S}) \otimes \bQ \cong \cO_{K,S}^\times\otimes\bQ$.
\end{lemma}

Now consider the case where $K=\bQ(\zeta_{l^m})$ is a cyclotomic field for a prime power $l^m > 2$. We have $[K:\bQ]=\varphi(l^m)$, $r_1=0$ and $r_2=\frac{\varphi(l^m)}{2}$. By \eqref{eq:Selparametrisation}, we can parametrise the depth-$N$ polylogarithmic Selmer scheme with coordinates $(x_i)_{1\leq i\leq d_1}, (y_i)_{1\leq i\leq d_1}, (z_{n,i})_{2\leq n\leq N,1\leq i\leq d_n}$. 
So the dimension of the global Selmer scheme is
\begin{align*}
    \dim_{\bQ_p} \mathrm{Sel}_{S,\PL,N}^{\mot}(X)_{\bQ_p} &= 2d_1+\sum_{n=2}^N d_n=2\#S+\left(\left\lceil \frac N2\right\rceil+1\right) r_1+(N+1)r_2-2\\
    &=2\#S+(N+1)\frac{\varphi(l^m)}{2}-2,
\end{align*}
whereas the local Selmer scheme $\prod_{\fp \mid p} \Pi_{\mathrm{PL},N,K_{\fp}}^{\dR}$ has the dimension $(N+1)[K:\bQ]=(N+1)\varphi(l^m)$. So we expect the Chabauty--Kim locus $X(\cO_K \otimes \bZ_p)_{S,\PL,N}$ to be defined by a system of $\dim_{\bQ_p} \prod_{\fp \mid p} \Pi_{\mathrm{PL},N,K_{\fp}}^{\dR} - \dim_{\bQ_p} \mathrm{Sel}_{S,\PL,N}^{\mot}(X)_{\bQ_p} = (N+1)\varphi(l^m)/2 +2-2\#S$ equations in $[K:\bQ]$ variables $z_1,\ldots,z_{[K:\bQ]}$.

\subsection{Producing points over cyclotomic fields}
\label{sec:producing-points}
Let $l$ be a prime, $m\geq 1$ such that $l^m >2$, $K=\bQ(\zeta_{l^m})$, and $S=\left\{(1-\zeta_{l^m})\right\}$. 
One of the reasons to focus on this case is that in the Chabauty--Kim method it is useful to have a large supply of $S$-integral points available. There are several ways to produce points in $X(\cO_{K,S})$. Observe that $1-\zeta_{l^m}^i \in \cO_{K,S}^{\times}$ for all integers~$i$ with $l^m \nmid i$, see e.g.\ \cite[Lemma~8]{siksek-visser}. Thus, the roots of unity $\zeta_{l^m}^i$ for $1 \leq i < l^m$ belong to $X(\cO_{K,S})$. 

\begin{lemma}
\label{lem:basicpts}
    $(\zeta_{l^m}^i-1)/(\zeta_{l^m}^j-1) \in X(\cO_{K,S})$ for any $1\leq i\ne j<l^m$.
\end{lemma}
\begin{proof}
    We have 
    \[ 1 - \frac{\zeta_{l^m}^i-1}{\zeta_{l^m}^j-1} = \zeta_{l^m}^i \frac{\zeta_{l^m}^{j-i}-1}{\zeta_{l^m}^j-1} \in \cO_{K,S}^{\times}. \qedhere \]
\end{proof}

Another way of producing points in $X(\cO_{K,S})$ is described in \cite{siksek-visser}. It uses the cyclotomic polynomials $\Phi_n(X)\coloneqq \prod_{i\in(\bZ/n\bZ)^\times}(X-\zeta_n^i) \in \bZ[X]$ and is based on the following observation.

\begin{lemma}
\label{cyclotomic-polynomial-S-unit}
    If $l^m \nmid n$ then $\Phi_n(\zeta_{l^m}) \in \cO_{K,S}^\times$. 
\end{lemma}

\begin{proof}
    Since $\Phi_n(X) \mid (X^n - 1)$, we have $\Phi_n(\zeta_{l^m}) \mid (\zeta_{l^m}^n - 1)$, which is an $S$-unit when $l^m \nmid n$, as observed above.
\end{proof}

Following \cite{siksek-visser}, say that a polynomial $F \in \bZ[X]$ is \emph{supercyclotomic} if it is of the form $X^r \Phi_{n_1}(X)\cdots\Phi_{n_k}(X)$. By \Cref{cyclotomic-polynomial-S-unit}, $F(\zeta_{l^m}) \in \cO_{K,S}^{\times}$ for such a polynomial when none of the $n_i$ is divisible by $l^m$. Assume we have a relation of the form $F -G = kH$ with $F,G,H$ supercyclotomic and $k$ a positive integer. Plugging in $\zeta_{l^m}$ yields
\[ \frac{F(\zeta_{l^m})}{k H(\zeta_{l^m})} - \frac{G(\zeta_{l^m})}{k H(\zeta_{l^m})} = 1, \]
which shows that $\frac{F(\zeta_{l^m})}{k H(\zeta_{l^m})} \in X(\cO_{K,S})$ whenever $F(\zeta_{l^m}), G(\zeta_{l^m}), H(\zeta_{l^m})$ and~$k$ are $S$-units. Examples of such ternary relations of supercyclotomic polynomials given in \cite{siksek-visser} are:
\begin{alignat*}{2}
    \Phi_2(X)^2-\Phi_3(X)&=X, \qquad&
    \Phi_2(X)^2-\Phi_4(X)&=2X,\\
    \Phi_2(X)^2-\Phi_1(X)^2&=4X, \qquad&
    \Phi_2(X)^4-\Phi_1(X)^4&=8X\Phi_4(X).
\end{alignat*}

In summary, we have the following methods to systematically produce points in $X(\cO_{K,S})$:
\begin{center}
\begin{tabularx}{\textwidth}{@{}cX@{}} 
 \textbf{Source 1:} & Roots of unity $\zeta_{l^m}^i$ for $1\leq i<l^m$; \\
 \textbf{Source 2:} & The points $(\zeta_{l^m}^i-1)/(\zeta_{l^m}^j-1)$ from \Cref{lem:basicpts}; \\
 \textbf{Source 3:} & Points from relations of supercyclotomic polynomials.
\end{tabularx}
\end{center}

To get further points, we can use the $\mathrm{Gal}(K/\bQ)$-action and the natural $S_3$-action on $\bP^1 \smallsetminus \{0,1,\infty\}$ generated by $z \mapsto 1/z$ and $z \mapsto 1-z$. The $S_3$-orbit of an $S$-integral point~$z$ is $\left\{z,1-z,\frac{1}{z},\frac{1}{1-z},\frac{z-1}{z},\frac{z}{z-1}\right\}$.

\begin{example}
    Let $K=\bQ(\zeta_{2^m})$ and $S=\left\{(1-\zeta_{2^m})\right\}$ where $m \geq 2$. Then $\varepsilon=\frac{\Phi_2(\zeta_{2^m})^2}{\zeta_{2^m}}$ and $\delta=-\frac{\Phi_3(\zeta_{2^m})}{\zeta_{2^m}}$ are $S$-integral points on $\bP^1\smallsetminus\left\{0,1,\infty\right\}$ with $\eps + \delta = 1$ by the first supercyclotomic polynomial relation above. Similarly, we get two points from the third relation, and when $m \geq 3$ we get two points each from the second and fourth relation.

    Let us take $m=3$ for a concrete example. Then eight points can be produced using the four relations:
\begin{alignat*}{3}
    \varepsilon_1&=\frac{\Phi_2(\zeta_8)^2}{\zeta_8}
    =2+\zeta_8-\zeta_8^3,\qquad
    &&\delta_1=-1-\zeta_8+\zeta_8^3; \\
    \varepsilon_2&=\frac{\Phi_2(\zeta_8)^2}{2\zeta_8}
    =1+\frac{1}{2}\zeta_8-\frac{1}{2}\zeta_8^3,\qquad
    &&\delta_2=-\frac{1}{2}\zeta_8+\frac{1}{2}\zeta_8^3;\\
    \varepsilon_3&=\frac{\Phi_2(\zeta_8)^2}{4\zeta_8}
    =\frac{1}{2}+\frac{1}{4}\zeta_8-\frac{1}{4}\zeta_8^3,\qquad
    &&\delta_3=\frac{1}{2}-\frac{1}{4}\zeta_8+\frac{1}{4}\zeta_8^3; \\
    \varepsilon_4&=\frac{\Phi_2(\zeta_8)^4}{8\zeta_8\Phi_4(\zeta_8)}
    =\frac{1}{2}+\frac{3}{8}\zeta_8-\frac{3}{8}\zeta_8^3,\qquad
    &&\delta_4=\frac{1}{2}-\frac{3}{8}\zeta_8+\frac{3}{8}\zeta_8^3.
\end{alignat*}

    To produce more points, consider the Galois action. Denote by $\sigma_i\in \mathrm{Gal}(K/\bQ)$ the automorphism sending $\zeta_8\mapsto -\zeta_8$. Then $\sigma_i\varepsilon_1=2-\zeta_8+\zeta_8^3$ gives a new point. Applying $\sigma_i$ to all eight points above, we obtain $12$ points in total.
    
    \begin{lemma}[{\cite[Lemma~15]{siksek-visser}}]
    Denote by $\sigma_+\in \mathrm{Gal}(\bQ(\zeta_{l^m})/\bQ)$ the complex conjugation. If $l^m\nmid n$, then
    $$\frac{\sigma_+\Phi_n(\zeta_{l^m})}{\Phi_n(\zeta_{l^m})}=\begin{cases}\zeta_{l^m}^{-\varphi(n)} & n\geq 2,\\
    -\zeta_{l^m}^{-1} & n=1.\end{cases}$$
    \end{lemma}
    
    Returning to our example $K=\bQ(\zeta_{8})$, we have $\sigma_+\Phi_2(\zeta_{8})=\zeta_8^{-1}\Phi_2(\zeta_8)$ by the lemma, and thus $\sigma_+\varepsilon_i=\varepsilon_i$ and $\sigma_+\delta_i=\delta_i$ for $i=1,2,3$. Similarly, we have $\sigma_+ \Phi_4(\zeta_8) = \zeta_8^{-2} \Phi_4(\zeta_8)$, which implies $\sigma_+\eps_4 = \eps_4$ and $\sigma_+ \delta_4 = \delta_4$. So the automorphism $\sigma_+$ does not produce any new solutions. 
    The $S_3$-action adds $12$ further points to those obtained from source 3, resulting in a total of 24 points. Source 2 contributes $42$ points, since the elements $
       (\zeta_8^i-1)/(\zeta_8^j-1)$ for $ 1\leq i\ne j<8$
    are pairwise distinct, and neither the Galois action nor the $S_3$-action produces new points. Source 1 (after Galois and $S_3$-action) produces $12$ points. 
    A direct computation shows that sources 2 and 3 are disjoint, while source 1 contributes exactly $3$ new points outside their union. Thus the three sources give $24+42+3=69$
    points in total. Since $\# X(\cO_{K,S})=75$, the remaining six points are
    \begin{gather*}
        -2+2\zeta_8-2\zeta_8^3,\quad 3-2\zeta_8+2\zeta_8^3,\quad 3+2\zeta_8-2\zeta_8^3,\\ -2-2\zeta_8+2\zeta_8^3,\quad \tfrac{1}{2}+\tfrac{1}{2}\zeta_8-\tfrac{1}{2}\zeta_8^3,\quad \tfrac{1}{2}-\tfrac{1}{2}\zeta_8+\tfrac{1}{2}\zeta_8^3.
    \end{gather*}
\end{example}

\section{Equations for \texorpdfstring{$K=\bQ(\zeta_8)$, $S=\left\{(1-\zeta_8)\right\}$}{K = Q(ζ8), S = \{(1 − ζ8)\}}}
Let $K=\bQ(\zeta_8)$, $S=\left\{(1-\zeta_8)\right\}$. Set 
$$\alpha\coloneqq 1+\zeta_8^2, \qquad \beta\coloneqq 1+\zeta_8+\zeta_8^{-1}=1+\sqrt{2}.$$
Then $\left\{\alpha,\beta\right\}$ is a $\bQ$-basis of $\cO_{K,S}^\times\otimes \bQ$. 
We have $\mathrm{Gal}(K/\bQ)=\mathrm{Gal}(K/\bQ(\zeta_8)^+)\times \mathrm{Gal}(K/\bQ(i))\simeq \bZ/2\bZ\times\bZ/2\bZ$ with generators the complex conjugation $\sigma_+\in \mathrm{Gal}(K/\bQ(\zeta_8)^+)$ and $\sigma_i\in \mathrm{Gal}(K/\bQ(i))$ with $\sigma_i(\zeta_8)=-\zeta_8$. Here, $\bQ(\zeta_8)^+ = \bQ(\sqrt{2})$ denotes the maximal totally real subfield of $K$.
Let $p$ be a rational prime which splits completely in~$K$, i.e.\ $p\equiv 1\pmod{8}$. Then there are four places above $p$ in $K$, i.e., $p\cO_K=\fp_1\fp_2\fp_3\fp_4$. Arrange the indices so that $\fp_2=\sigma_+^*\fp_1$, $\fp_3=\sigma_i^*\fp_1$, and $\fp_4=(\sigma_+\sigma_i)^*\fp_1$. 
\subsection{The Galois action on the motivic Lie algebra}
The degree-$(-n)$ part of the abelianisation of $L_S^{\MT}$ is isomorphic to $(K_{2n-1}(\cO_{K,S})\otimes\bQ)^\vee$, so the first step in studying the Galois action on $L_S^{\MT}$ is to study the Galois action on rational $K$-groups. 
By \Cref{lem:Kdim}, $\dim_{\bQ} K_{2n-1}(\cO_{K,S})\otimes\bQ=2$ for all $n\geq 1$ in our setting.

\begin{lemma}
\label{lem:galoislem}
    For all $n \geq 2$ there exist $\sigma_{n,1}, \sigma_{n,2} \in (L_S^{\MT})_{-n}$ mapping to a basis of $(L_S^{\MT})_{-n}^{\ab}$ satisfying
    \begin{equation}
        \sigma_+^* \sigma_{n,j} = (-1)^{n+1} \sigma_{n,j}, \quad \sigma_i^* \sigma_{n,j} = (-1)^{j+1} \sigma_{n,j} \qquad \text{for $j=1,2$.}
    \end{equation}
\end{lemma}

\begin{proof}
For a subfield $L$ of $K$, we have \cite[Eq.~(2.16.2)]{deligne-goncharov}
\[ K_{2n-1}(L) \otimes \bQ \cong (K_{2n-1}(K) \otimes \bQ)^{\Gal(K/L)}. \]
In our case $K=\bQ(\zeta_8)$. If we take $L=\bQ(i)$, then 
\begin{align*}
    (K_{2n-1}(K) \otimes \bQ)^{\sigma_i}&\simeq K_{2n-1}(\bQ(i))\otimes \bQ\\
    \text{ has $\bQ$-dimension}&=\begin{cases}r_1(\bQ(i))+r_2(\bQ(i))=1, & n\geq 3\text{ odd},\\
r_2(\bQ(i))=1, & n\geq 2\text{ even}.\end{cases}
\end{align*}
So there is a $\bQ$-basis $\tilde{\sigma}_{n,j}$ ($j=1,2$) of $(K_{2n-1}(K) \otimes \bQ)^{\vee}$ such that $\sigma_i^*\tilde{\sigma}_{n,j}=(-1)^{j+1}\tilde{\sigma}_{n,j}$ for $n\geq 2$, $j=1,2$.

If we take $L=\bQ(\zeta_8)^+$, then 
\begin{align*}
    (K_{2n-1}(K) \otimes \bQ)^{\sigma_+}&\simeq K_{2n-1}(\bQ(\zeta_8)^+)\otimes \bQ\\
    \text{ has $\bQ$-dimension}&=\begin{cases}r_1(\bQ(\zeta_8)^+)+r_2(\bQ(\zeta_8)^+)=2, & n\geq 3\text{ odd},\\
r_2(\bQ(\zeta_8)^+)=0, & n\geq 2\text{ even}.\end{cases}
\end{align*}
So $\sigma_+$ acts on $K_{2n-1}(K) \otimes \bQ$ (and hence its dual) as the identity for odd $n\geq 3$, and as multiplication by $-1$ for even $n\geq 2$. That means, $\sigma_+^*\tilde{\sigma}_{n,j}=(-1)^{n+1}\tilde{\sigma}_{n,j}$ for $n\geq 2$, $j=1,2$.

Finally, since $(L_S^{\MT})_{-n}\twoheadrightarrow (L_S^{\MT})_{-n}^{\ab}\simeq (K_{2n-1}(K) \otimes \bQ)^{\vee}$ is a $\Gal(K/\bQ)$-equivariant map, we may choose a $\Gal(K/\bQ)$-equivariant splitting by Maschke's Theorem, which provides the desired lifts $\sigma_{n,j}\in (L_S^{\MT})_{-n}$ of $\tilde{\sigma}_{n,j}\in (L_S^{\MT})_{-n}^{\ab}$ preserving the Galois relations, i.e., $\sigma_+^* \sigma_{n,j} = (-1)^{n+1} \sigma_{n,j}$ and $\sigma_i^* \sigma_{n,j} = (-1)^{j+1}\sigma_{n,j}$ for $n\geq 2$, $j=1,2$.
\end{proof}

\subsection{Deriving equations in depth 2}
\label{subsec:eqdepth2}
$L_S^{\MT}$ is the free graded pro-nilpotent Lie algebra generated by $\tau_\alpha$, $\tau_\beta$ in degree $-1$ (which are dual to $\log^{\fu}(\alpha),\log^{\fu}(\beta)$ and ordered as $(\tau_{\beta},\tau_{\alpha})$), and $\sigma_{n,1}$, $\sigma_{n,2}$ in degree $-n$ for each $n\geq 2$. 
Write $\chi_1:=\sigma_{2,1}$ and $\chi_2:=\sigma_{2,2}$. 
Then, by \Cref{lem:galoislem} and direct computation of the Galois action on $\cO_{K,S}^{\times} \otimes \bQ$, the action on the chosen generators of $L_S^{\MT}$ is given as follows:
\begin{align*}
\sigma_+^*\tau_\alpha=\tau_\alpha,\quad 
\sigma_i^*\tau_\alpha=\tau_\alpha,\quad &
\sigma_+^*\tau_\beta=\tau_\beta,\quad 
\sigma_i^*\tau_\beta= -\tau_\beta,\\
\sigma_+^*\chi_1=-\chi_1,\quad 
\sigma_i^*\chi_1=\chi_1,\quad &
\sigma_+^*\chi_2=-\chi_2,\quad 
\sigma_i^*\chi_2=-\chi_2.
\end{align*}

By Eqs.~(\ref{eq:tauimapto}) and~(\ref{eq:sigmanmapto}), coordinates $x_{\alpha}, y_{\alpha}, x_{\beta}, y_{\beta}, z_1, z_2$ on the depth-2 Selmer scheme are given by
\begin{equation}
\label{eq:depth2-parametrisation}
\begin{alignedat}{2}
\xi(\tau_\alpha)&=x_\alpha(\xi)e_0+y_\alpha(\xi)e_1, & \qquad
\xi(\tau_\beta)&=x_\beta(\xi)e_0+y_\beta(\xi)e_1,\\
\xi(\chi_1)&=z_1(\xi)[e_0,e_1], & \qquad
\xi(\chi_2)&=z_2(\xi)[e_0,e_1].
\end{alignedat}
\end{equation}

The localisation map $\loc_p$ is the product of the evaluation maps at four places $\mathrm{ev}_{\eps_1}\times \cdots \times \mathrm{ev}_{\eps_4}$. Write the $\fp_1$-adic period point
\begin{equation}
\label{eq:firsteta1}
  \eps_1\coloneqq \varepsilon_{\fp_1}=c_{\tau_\alpha}\tau_\alpha+c_{\tau_\beta}\tau_\beta+c_{\chi_1}\chi_1+c_{\chi_2}\chi_2+c_{\tau_\alpha\tau_\beta}[\tau_\alpha,\tau_\beta]+(\text{degree $<-2$})  
\end{equation}
with $c_{\tau_\alpha},c_{\tau_\beta},c_{\chi_1},c_{\chi_2},c_{\tau_\alpha\tau_\beta}\in\bQ_p$. Since we are at depth $2$, we can neglect the terms of degree $<-2$.
We know that $c_{\tau_\alpha}=\log(\alpha)$, $c_{\tau_\beta}=\log(\beta)$ by \Cref{lem:logalpha}, using the embedding $K \hookrightarrow K_{\fp_1} = \bQ_p$ for the $p$-adic logarithms. We now determine the value of $c_{\tau_\alpha\tau_\beta}$. 

Note that $\beta$ is an $S$-integral point with $1- \beta = \zeta_8^3 \alpha$. Let $\xi_\beta$ denote its image under the global Kummer map $j_S$ in the global Selmer scheme. By \Cref{lem:globalKummervaluation}, 
$$x_{\alpha}(\xi_\beta)=0,\quad y_{\alpha}(\xi_\beta)=-1,\quad x_\beta(\xi_\beta)=1,\quad y_\beta(\xi_\beta)=0.$$
Thus,
$$\xi_\beta(\tau_\alpha)=-e_1,\quad \xi_\beta(\tau_\beta)=e_0,\quad \xi_\beta([\tau_\alpha,\tau_\beta])=[e_0,e_1].$$
Evaluating $\xi_\beta$ on both sides of \eqref{eq:firsteta1} gives
$$\log(\beta)e_0+L_1(\beta)e_1+L_2(\beta)[e_0,e_1]=\log(\beta)e_0-\log(\alpha)e_1+(c_{\chi_1}z_1(\xi_\beta)+c_{\chi_2}z_2(\xi_\beta)+c_{\tau_\alpha\tau_\beta})[e_0,e_1].$$
Comparing the coefficients of $[e_0,e_1]$ gives
\begin{equation}
\label{eq:cL2beta1}
    L_2(\beta)=c_{\chi_1}z_1(\xi_\beta)+c_{\chi_2}z_2(\xi_\beta)+c_{\tau_\alpha\tau_\beta}.
\end{equation}
We claim that $z_1(\xi_{\beta}) = z_2(\xi_{\beta}) = 0$. Indeed, $\beta = 1 + \sqrt{2}$ is fixed by complex conjugation, hence $\xi_{\beta} = (\sigma_+)_*(\xi_{\beta}) = \xi_{\beta} \circ \sigma_+^*$ by Galois equivariance of the global Kummer map and \Cref{thm:GalactionmotivicSel}. On the other hand, $\sigma_+^* \chi_j = -\chi_j$ for $j=1,2$, hence
\[ \xi_{\beta}(\chi_j) = (\xi_{\beta} \circ \sigma_+^*)(\chi_j) = -\xi_{\beta}(\chi_j), \]
which implies $\xi_{\beta}(\chi_j) = 0$ and thus $z_j(\xi_{\beta}) = 0$ for $j=1,2$ as claimed. Now \eqref{eq:cL2beta1} yields
\begin{equation}
\label{eq:ctaualphataubeta}
   c_{\tau_\alpha\tau_\beta}=L_2(\beta). 
\end{equation}
As a result, the expansion of $\eps_1$ in the chosen generators of $L_S^{\MT}$ is given by
$$\eps_1 = \log(\alpha)\tau_\alpha+\log(\beta)\tau_\beta+c_{\chi_1}\chi_1+c_{\chi_2}\chi_2+L_2(\beta)[\tau_\alpha,\tau_\beta] + (\text{degree $<-2$}).$$

Recall we arranged the indices so that $\fp_2=\sigma_+^*\fp_1$, $\fp_3=\sigma_i^* \fp_1$, $\fp_4=(\sigma_+\sigma_i)^*\fp_1$. By \Cref{thm:Galpermuteseta}, $\eps_2=\sigma_+^*\eps_1$, $\eps_3=\sigma_i^* \eps_1$, $\eps_4=(\sigma_+\sigma_i)^* \eps_1$. So we apply the Galois action to get the $\fp_i$-adic period points for $i=2,3,4$ up to depth $2$:
\begin{align*}
\eps_2&\coloneqq \varepsilon_{\fp_2}=\log(\alpha)\tau_\alpha+\log(\beta)\tau_\beta-c_{\chi_1}\chi_1-c_{\chi_2}\chi_2+L_2(\beta)[\tau_\alpha,\tau_\beta] + \ldots ,\\
\eps_3&\coloneqq \varepsilon_{\fp_3}=\log(\alpha)\tau_\alpha-\log(\beta)\tau_\beta+c_{\chi_1}\chi_1-c_{\chi_2}\chi_2-L_2(\beta)[\tau_\alpha,\tau_\beta] + \ldots,\\
\eps_4&\coloneqq \varepsilon_{\fp_4}=\log(\alpha)\tau_\alpha-\log(\beta)\tau_\beta-c_{\chi_1}\chi_1+c_{\chi_2}\chi_2-L_2(\beta)[\tau_\alpha,\tau_\beta] + \ldots,
\end{align*}
where the ellipses represent elements in degree $< -2$.

As a consequence, the evaluation maps $\ev_{\eps_i}$ in depth~2 are given by
\begin{align*}
\mathrm{ev}_{\eps_1}(\xi)&=(\log(\alpha)x_\alpha(\xi)+\log(\beta)x_\beta(\xi))e_0+(\log(\alpha)y_\alpha(\xi)+\log(\beta)y_\beta(\xi))e_1\\ &+
\left(c_{\chi_1}z_1(\xi)+c_{\chi_2}z_2(\xi)+L_2(\beta)(x_\alpha(\xi)y_\beta(\xi)-x_\beta(\xi)y_\alpha(\xi))\right)[e_0,e_1],
\\
\mathrm{ev}_{\eps_2}(\xi)&=
(\log(\alpha)x_\alpha(\xi)+\log(\beta)x_\beta(\xi))e_0+(\log(\alpha)y_\alpha(\xi)+\log(\beta)y_\beta(\xi))e_1\\ &+
\left(-c_{\chi_1}z_1(\xi)-c_{\chi_2}z_2(\xi)+L_2(\beta)(x_\alpha(\xi)y_\beta(\xi)-x_\beta(\xi)y_\alpha(\xi))\right)[e_0,e_1],\\
\mathrm{ev}_{\eps_3}(\xi)&=(\log(\alpha)x_\alpha(\xi)-\log(\beta)x_\beta(\xi))e_0+(\log(\alpha)y_\alpha(\xi)-\log(\beta)y_\beta(\xi))e_1\\ &+
\left(c_{\chi_1}z_1(\xi)-c_{\chi_2}z_2(\xi)-L_2(\beta)(x_\alpha(\xi)y_\beta(\xi)-x_\beta(\xi)y_\alpha(\xi))\right)[e_0,e_1],\\
\mathrm{ev}_{\eps_4}(\xi)&=(\log(\alpha)x_\alpha(\xi)-\log(\beta)x_\beta(\xi))e_0+(\log(\alpha)y_\alpha(\xi)-\log(\beta)y_\beta(\xi))e_1\\
&+
\left(-c_{\chi_1}z_1(\xi)+c_{\chi_2}z_2(\xi)-L_2(\beta)(x_\alpha(\xi)y_\beta(\xi)-x_\beta(\xi)y_\alpha(\xi))\right)[e_0,e_1].
\end{align*}

Write $\ev_{\varepsilon_i}(\xi)=(X_i,Y_i,Z_i)$, corresponding to $(\log(z_i),L_1(z_i),L_2(z_i))$ in the local coordinates for $i=1,\ldots,4$. From the above we have
$X_1=X_2$, $X_3=X_4$, $Y_1=Y_2$, $Y_3=Y_4$, $Z_1+Z_2=-Z_3-Z_4$, and 
\begin{equation}
\label{eq:xalphaxbeta}
    x_\alpha(\xi)=\frac{X_1+X_3}{2\log(\alpha)},\quad x_\beta(\xi)=\frac{X_1-X_3}{2\log(\beta)},\quad
    y_\alpha(\xi)=\frac{Y_1+Y_3}{2\log(\alpha)},\quad y_\beta(\xi)=\frac{Y_1-Y_3}{2\log(\beta)},
\end{equation}
and the last equation becomes
\begin{equation}
\label{eq:c2}
Z_1+Z_2=c_2\left(-X_3Y_1+X_1Y_3\right) \quad \text{ where }\quad c_2=-\frac{L_2(\beta)}{\log(\alpha)\log(\beta)}.
\end{equation}

If we assume $c_{\chi_1},c_{\chi_2}\ne 0$ (which is implied by the $p$-adic period conjecture by \Cref{period-conjecture-implies-nonvanishing}), then the equations above precisely cut out the scheme-theoretic image of the localisation map. In summary, we have the following.
\begin{theorem}
\label{thm:depth2-equations}
Let $K=\bQ(\zeta_8)$ and $S=\left\{(1-\zeta_8)\right\}$. The depth-2 polylogarithmic Chabauty--Kim locus $X(\cO_K\otimes \bZ_p)_{S,\PL,2}$ is contained in the solution set of the following equations (with the coordinates $(X_i,Y_i,Z_i)=(\log(z_i),L_1(z_i),L_2(z_i))$ for $i=1,2,3,4$):
\begin{equation}
\label{eq:keyeq}
\begin{cases}
    X_1=X_2,\quad X_3=X_4,\\
    Y_1=Y_2,\quad\ Y_3=Y_4,\\
    Z_1+Z_2=-Z_3-Z_4,\\
    Z_1+Z_2=c_2\left(-X_3Y_1+X_1Y_3\right),
\end{cases}
\end{equation}
where $c_2$ is defined in \eqref{eq:c2}. The containment is an equality if $c_{\chi_1},c_{\chi_2}\ne 0$.
\end{theorem}
\begin{lemma}
\label{lem:cchiinonzero}
    For $i=1,2$, if $\Li_2(\zeta_8)+(-1)^{i+1} \Li_2(-\zeta_8)\ne 0$, then $c_{\chi_i}\ne 0$. This is the case for all $p<200$ which split completely in $\bQ(\zeta_8)$. In particular, the equality result in \Cref{thm:depth2-equations} holds for these~$p$.
\end{lemma}
\begin{proof}
    By the formula of the evaluation maps above, $c_{\chi_1}=0$ implies $Z_1-Z_2+Z_3-Z_4=4c_{\chi_1}z_1(\xi)=0$, and $c_{\chi_2}=0$ implies $Z_1-Z_2-Z_3+Z_4=4c_{\chi_2}z_2(\xi)=0$. Combining with $Z_1+Z_2=-Z_3-Z_4$ in \eqref{eq:keyeq}, the additional equation $Z_1+Z_3=0$, respectively $Z_1+Z_4=0$, holds on the Chabauty--Kim locus and hence on $X(\cO_{K,S})$. Plugging in $z=\zeta_8\in X(\cO_{K,S})$ gives $\Li_2(\zeta_8)\pm \Li_2(-\zeta_8)=0$, which completes the proof of the first part. See \url{https://github.com/martinluedtke/PolylogNF} for Sage code verifying the non-vanishing for $p < 200$. 
\end{proof}

\subsection{Refined Chabauty--Kim loci}
\label{sec:refined}

The \emph{refined Selmer scheme} $\Sel_{S,\PL,N}^{\min}(X)$ is the closed subscheme of $\Sel_{S,\PL,N}(X)$ defined by the equations $x_{\fl}y_{\fl}(x_{\fl}+y_{\fl})=0$ for $\fl\in S$ \cite[Definition 4.16]{LL:PolylogNF}, where the functions $x_{\fl}$ and $y_{\fl}$ are defined via
$\xi(\tau_{\fl}) = x_{\fl}(\xi)e_0 + y_{\fl}(\xi)e_1$
for graded cocycles $\xi$, and $\tau_{\fl} \in (L_S^{\MT})_{-1}$ is given by $\langle \tau_{\fl}, \log^{\fu}(x)\rangle = v_{\fl}(x)$ for $x \in \cO_{K,S}^{\times} \otimes \bQ$. In our case, $S=\left\{\fl\right\}$ with $\fl=(1-\zeta_8)$, so the refined Selmer scheme has three components $\Sel_{S,\PL,N}^{\min}=\bigcup_{\Sigma\in\left\{0,1,\infty\right\}}\Sel_{S,\PL,N}^{\Sigma}$, defined by $y_{\fl}=0$, $x_{\fl}=0$, $x_{\fl}+y_{\fl}=0$, respectively. We get the $\Sigma$-refined Chabauty--Kim loci $X(\cO_K\otimes\bZ_p)_{S,\PL,N}^{\Sigma}$, and their union is the \emph{refined Chabauty--Kim locus} $X(\cO_K\otimes\bZ_p)_{S,\PL,N}^{\min}$,
which still contains $X(\cO_{K,S})$. 
Since $v_{\fl}(\alpha) = 2$, we have $\tau_{\fl} = 2 \tau_{\alpha}$, hence $x_{\fl} = 2x_{\alpha}$ and $y_{\fl} = 2y_{\alpha}$. For $\Sigma=1$, the condition $x_\fl=0$ is equivalent to $x_\alpha=0$, so by \eqref{eq:xalphaxbeta}, the additional equation $X_1+X_3=0$ holds on $X(\cO_K\otimes\bZ_p)_{S,\PL,N}^{(1)}$. Similarly, for $\Sigma = 0$ ($y_{\alpha} = 0$), $Y_1+Y_3=0$ holds on $X(\cO_K\otimes\bZ_p)_{S,\PL,N}^{(0)}$; and for $\Sigma=\infty$ ($x_\alpha+y_\alpha=0$), $X_1+X_3+Y_1+Y_3=0$ holds on $X(\cO_K\otimes\bZ_p)_{S,\PL,N}^{(\infty)}$.

\subsection{Solving equations in depth 2}
To solve the equations \eqref{eq:keyeq}, we first notice that
\begin{equation*}
\begin{cases}
    X_1=X_2\\
    Y_1=Y_2
    \end{cases}\Leftrightarrow \quad 
\begin{cases}
    \log(z_1)=\log(z_2)\\
    \log(1-z_1)=\log(1-z_2).
\end{cases}
\end{equation*}

By \cite[Lemma~8.4]{LL:PolylogNF}, the solutions of these equations in $X(\bZ_p) \times X(\bZ_p)$ are those pairs $(z_1,z_2)$ with $z_1=z_2$ or 
\begin{equation}
\label{eq:root-of-unity-parametrisation}
    (z_1,z_2)=\left(\frac{1-\eta}{\zeta-\eta},\zeta\frac{1-\eta}{\zeta-\eta}\right)
\end{equation}
for roots of unity $\zeta,\eta \in \bZ_p$ with $\zeta,\eta\ne 1$ and $\zeta\ne \eta$. In particular, there are only finitely many solution pairs $(z_1,z_2)$ away from the diagonal.
The same discussion also applies to $(z_3,z_4)$. Thus, we have four types of solutions $(z_1,z_2,z_3,z_4)$ of the equations \eqref{eq:keyeq}, depending on whether $(z_1,z_2)$ and $(z_3,z_4)$ are on the diagonal or off-diagonal:
\begin{center}
\begin{tabularx}{\textwidth}{ |c|c|X| } 
 \hline
 type & description & remark \\ \hline
 off-off & $z_1\ne z_2$, $z_3\ne z_4$ & There are only finitely many solutions. \\ \hline
 on-off  & $z_1=z_2$, $z_3\ne z_4$   & \multirow{2}{=}{The system can be reduced to several instances of an equation in one variable. These can be effectively solved by Hensel lifting.} \\ \cline{1-2}
 off-on  & $z_1\ne z_2$, $z_3=z_4$   & \\ \hline
 on-on   & $z_1=z_2$, $z_3=z_4$   & This results in two equations in two variables. \\ \hline
\end{tabularx}
\end{center}

The two equations in two variables $z_1, z_3$ for the on-on case are 
\begin{equation}
\label{eq:onontwovartwoeq}
    \begin{cases}
        L_2(z_1)=-L_2(z_3)\\
        2\log(\alpha)\log(\beta)L_2(z_1)=-L_2(\beta)(\log(z_3)\log(1-z_1)-\log(z_1)\log(1-z_3)).
    \end{cases}
\end{equation}
We need the multivariate version of Hensel's lemma to solve the system of two equations in two variables. An implementation in SageMath by Francesca Bianchi is available at \url{https://github.com/steffenmueller/LinQC}. For the \emph{refined} on-on locus, we have one additional equation for each $\Sigma \in \{0,1,\infty\}$, reducing us back to equations in one variable.

The statistics below show the number of solutions among four types for various situations for $p=17$. Recall that $\# X(\cO_{K,S})=75$, out of which $\# X(\cO_{K,S})^{(1)}=25$ are guaranteed to lie in the $(1)$-refined locus.

\resizebox{\textwidth}{!}{
\begin{tabular}{ |c|c|c|c|c|c|c| } 
 \hline
 type & \# sols & \# exceptional sols & \# refined sols & \# refined except. sols & \# refined sols at $\Sigma=1$ & \# refined except. sols at $\Sigma=1$
  \\  \hline
 off-off & $636$ & $594$ & $636$ & $594$ & $212$ & $198$   \\  
 on-off & $42$ & $42$ & $42$ & $42$ & $14$ & $14$ \\ 
 off-on & $42$ & $42$ & $42$ & $42$ & $14$ & $14$ \\ 
 on-on & $183$ & $150$ & $63$ & $30$ & $21$ & $10$ \\ \hline
  TOTAL & $903$ & $828$ & $783$ & $708$ & $261$ & $236$ \\ \hline
\end{tabular}}

Next we give a theoretical description of the depth $2$ solutions. The system of equations \eqref{eq:keyeq} has the swapping symmetry $z_1\leftrightarrow z_2$, and $z_3\leftrightarrow z_4$. 
Let $s_{34}(z_1,z_2,z_3,z_4)\coloneqq (z_1,z_2,z_4,z_3)$. The simultaneous swaps $z_1\leftrightarrow z_2$ and $z_3\leftrightarrow z_4$ are induced by a Galois automorphism, so $X(\cO_{K,S})$ is stable under their simultaneous action. Thus its closure under the two individual swapping symmetries is $\operatorname{Sw}\bigl(X(\cO_{K,S})\bigr)\coloneqq X(\cO_{K,S})\cup s_{34}\bigl(X(\cO_{K,S})\bigr)$. Define $g_0(t)=1-t,\ g_1(t)=t,\ g_\infty(t)=\frac{1}{1-t}$, and set
\begin{align*}
    \mathcal{G}_p & \coloneqq  \bigcup_{g\in\{g_0,g_1,g_\infty\}} \bigcup_{\zeta,\eta\in\mu_{p-1}(\bZ_p)\setminus\{1\}} \left\{ \bigl(g(\zeta),g(\zeta^{-1}), g(\eta),g(\eta^{-1})\bigr) \right\},\\
    \mathcal{L}_p & \coloneqq \left\{ (\zeta_6^{\pm 1},\zeta_6^{\mp 1},z,z),\ (z,z,\zeta_6^{\pm 1},\zeta_6^{\mp 1}) : \ L_2(z)=0 \right\} \quad (\text{set to be empty if $\zeta_6\notin \bZ_p$}).
\end{align*}

\begin{theorem}
\label{thm:depth2-locus-description}
Assume the $p$-adic period conjecture. For $K=\bQ(\zeta_8)$, $S=\left\{(1-\zeta_8)\right\}$ and totally split primes $p$,
\begin{align}
\label{eq:descriptionofdepth2sols}
\operatorname{Sw}\bigl(X(\cO_{K,S})\bigr) \cup\mathcal{G}_p\cup\mathcal{L}_p\cup X(\bZ[\sqrt{2}]\otimes\bZ_p)_{(\sqrt{2}),\PL,2} &\subseteq X(\cO_K\otimes\bZ_p)_{S,\mathrm{PL},2}, \\
\label{eq:descriptionofdepth2ref}
\operatorname{Sw}\bigl(X(\cO_{K,S})\bigr) \cup \mathcal{G}_p\cup X(\bZ[\sqrt{2}]\otimes\bZ_p)_{(\sqrt{2}),\PL,2}^{\min} &\subseteq X(\cO_K\otimes\bZ_p)_{S,\mathrm{PL},2}^{\min}.  
\end{align}
All of the above containments are equalities unconditionally for all primes $p<200$ which split completely in $K$.
\end{theorem}
\begin{proof}
    By the discussion in the end of §\ref{subsec:eqdepth2}, the refined equations in depth $2$ are given by
    $$\eqref{eq:keyeq} \text{ and } (X_1+X_3=0 \text{ or } Y_1+Y_3=0 \text{ or } X_1+X_3+Y_1+Y_3=0),$$
    which are stable under the swap $z_3\leftrightarrow z_4$. So $\mathrm{Sw}(X(\cO_{K,S}))\subseteq X(\cO_K\otimes\bZ_p)_{S,\mathrm{PL},2}^{\min}$.

    By the functional equations $L_2(z)+L_2(z^{-1})=0$ and $L_2(z)+L_2(1-z)=0$ \cite[Lemmas~4.5, 4.6]{LL:PolylogNF}, we get $L_2(g(z))+L_2(g(z^{-1}))=0$, and hence $(g(\zeta),g(\zeta^{-1}),g(\eta),g(\eta^{-1}))$ satisfies \eqref{eq:keyeq} for $g=g_0,g_1,g_\infty$. Checking the additional refined equations, we also find that $(g_i(\zeta),g_i(\zeta^{-1}),g_i(\eta),g_i(\eta^{-1}))\in X(\cO_K\otimes\bZ_p)_{S,\PL,2}^{(i)}$ for $i=0,1,\infty$; thus $\mathcal{G}_p\subseteq X(\cO_K\otimes\bZ_p)_{S,\PL,2}^{\min}$. The inclusion $X(\bZ[\sqrt{2}]\otimes\bZ_p)_{(\sqrt{2}),\PL,2}^{(\min)} \subseteq X(\cO_K\otimes\bZ_p)_{S,\mathrm{PL},2}^{(\min)}$ follows from the functoriality of Chabauty--Kim loci in field extensions \cite[Lemma~2.5]{LL:PolylogNF}; it is given by $(z_1,z_2)\mapsto (z_1,z_1,z_2,z_2)$. This shows \eqref{eq:descriptionofdepth2ref}.

    It is easy to check that the tuples in $\mathcal{L}_p$ satisfy \eqref{eq:keyeq} to get the inclusion \eqref{eq:descriptionofdepth2sols}. Finally, verifying that the containments are equalities for $p<200$ is done in the Sage code \url{https://github.com/martinluedtke/PolylogNF}. The equalities hold unconditionally (i.e., independent of the conjecture) since we checked $c_{\chi_i}\ne 0$ for $p<200$ in \Cref{lem:cchiinonzero}.
\end{proof}

\subsection{Depth 3}
\label{subsection:depth3}
We list the basis elements in the degree $-3$ part of $L_S^{\MT}$ and its Galois action in the table below.
\begin{center}
\begin{tabular}{ |c|c|c|c|c|c|c|c|c| } 
 \hline
  $\mathrm{id}$ & $[\tau_\alpha,[\tau_\alpha,\tau_\beta]]$ & $[\tau_\beta,[\tau_\alpha,\tau_\beta]]$ & $[\tau_\alpha,\chi_1]$ &
  $[\tau_\alpha,\chi_2]$ &
  $[\tau_\beta,\chi_1]$ &
  $[\tau_\beta,\chi_2]$ &
  $\sigma_{3,1}$ & $\sigma_{3,2}$
  \\  \hline
  $\sigma_+$ & $1$ & $1$ & $-1$ & $-1$ & 
  $-1$ & $-1$ & $1$ & $1$ 
  \\ \hline
    $\sigma_i$ & $-1$ & $1$ & $1$ & $-1$ & $-1$
   & $1$ & $1$ &  $-1$
  \\ \hline
\end{tabular}
\end{center}

Let $\kappa_1,\kappa_2\in (L^{\MT}_{S})_{-3} \otimes \bQ_p$ with $(\sigma_+,\sigma_i)\kappa_1=(1,1)\kappa_1$ and $(\sigma_+,\sigma_i)\kappa_2=(1,-1)\kappa_2$. For the computation, let all terms in the expansion of $\eps_1$ with $(\sigma_+,\sigma_i)=(1,1)$ (such as $[\tau_\beta,[\tau_\alpha,\tau_\beta]]$ and $\sigma_{3,1}$) be absorbed in $\kappa_1$, and do the same thing for $\kappa_2$. 
Keeping the depth-$2$ Selmer scheme coordinates from \eqref{eq:depth2-parametrisation}, we write the new coordinates appearing in depth $3$ as
\begin{equation}
\label{eq:depth3-parametrisation}
\begin{alignedat}{2}
\xi(\kappa_1) &= k_1(\xi)[e_0,[e_0,e_1]], & \qquad
\xi(\kappa_2) &= k_2(\xi)[e_0,[e_0,e_1]].
\end{alignedat}
\end{equation}

Write the $\fp_1$-adic period point (modulo degree $<-3$):
\begin{align*}
\varepsilon_1= & \log(\alpha)\tau_\alpha+\log(\beta)\tau_\beta+c_{\chi_1}\chi_1+c_{\chi_2}\chi_2+c_{\tau_\alpha\tau_\beta}[\tau_\alpha,\tau_\beta]\\
& +c_{\tau_\alpha\chi_1}[\tau_\alpha,\chi_1]+c_{\tau_\alpha\chi_2}[\tau_\alpha,\chi_2]+c_{\tau_\beta\chi_1}[\tau_\beta,\chi_1]+c_{\tau_\beta\chi_2}[\tau_\beta,\chi_2]+c_{\kappa_1}\kappa_{1}+c_{\kappa_2}\kappa_{2}.
\end{align*}
where $c_{\tau_\alpha\tau_\beta}=L_2(\beta)$ and $c_{\tau_\alpha\chi_1},c_{\tau_\alpha\chi_2},c_{\tau_\beta\chi_1},c_{\tau_\beta\chi_2},c_{\kappa_1},c_{\kappa_2}$ are undetermined coefficients.
Applying the Galois action, we get the expansions of the $\fp_i$-adic period points for $i=2,3,4$ (moduli degree $<-3$):
\begin{align*}
\varepsilon_2=&\log(\alpha)\tau_\alpha+\log(\beta)\tau_\beta-c_{\chi_1}\chi_1-c_{\chi_2}\chi_2+c_{\tau_\alpha\tau_\beta}[\tau_\alpha,\tau_\beta]\\
&-c_{\tau_\alpha\chi_1}[\tau_\alpha,\chi_1]-c_{\tau_\alpha\chi_2}[\tau_\alpha,\chi_2]-c_{\tau_\beta\chi_1}[\tau_\beta,\chi_1]-c_{\tau_\beta\chi_2}[\tau_\beta,\chi_2]+c_{\kappa_1}\kappa_{1}+c_{\kappa_2}\kappa_{2},\\
\varepsilon_3=&\log(\alpha)\tau_\alpha-\log(\beta)\tau_\beta+c_{\chi_1}\chi_1-c_{\chi_2}\chi_2-c_{\tau_\alpha\tau_\beta}[\tau_\alpha,\tau_\beta]\\
&+c_{\tau_\alpha\chi_1}[\tau_\alpha,\chi_1]-c_{\tau_\alpha\chi_2}[\tau_\alpha,\chi_2]-c_{\tau_\beta\chi_1}[\tau_\beta,\chi_1]+c_{\tau_\beta\chi_2}[\tau_\beta,\chi_2]+c_{\kappa_1}\kappa_{1}-c_{\kappa_2}\kappa_{2},\\
\varepsilon_4=&\log(\alpha)\tau_\alpha-\log(\beta)\tau_\beta-c_{\chi_1}\chi_1+c_{\chi_2}\chi_2-c_{\tau_\alpha\tau_\beta}[\tau_\alpha,\tau_\beta]\\
&-c_{\tau_\alpha\chi_1}[\tau_\alpha,\chi_1]+c_{\tau_\alpha\chi_2}[\tau_\alpha,\chi_2]+c_{\tau_\beta\chi_1}[\tau_\beta,\chi_1]-c_{\tau_\beta\chi_2}[\tau_\beta,\chi_2]+c_{\kappa_1}\kappa_{1}-c_{\kappa_2}\kappa_{2}.
\end{align*}
As a consequence, the evaluation maps $\ev_{\eps_i}$ are given by
\begin{align*}
\mathrm{ev}_{\varepsilon_1}(\xi)&=(\log(\alpha)x_\alpha(\xi)+\log(\beta)x_\beta(\xi))e_0+(\log(\alpha)y_\alpha(\xi)+\log(\beta)y_\beta(\xi))e_1\\ &+
\left(c_{\chi_1}z_1(\xi)+c_{\chi_2}z_2(\xi)+c_{\tau_\alpha\tau_\beta}(x_\alpha(\xi)y_\beta(\xi)-x_\beta(\xi)y_\alpha(\xi))\right)[e_0,e_1]\\ &+
\Bigl((c_{\kappa_1}k_1(\xi)+c_{\kappa_2}k_2(\xi)+c_{\tau_\alpha\chi_1}x_\alpha(\xi)z_1(\xi)+c_{\tau_\alpha\chi_2}x_\alpha(\xi)z_2(\xi) \\
&\quad + c_{\tau_\beta\chi_1}x_\beta(\xi)z_1(\xi)+c_{\tau_\beta\chi_2}x_\beta(\xi)z_2(\xi)\Bigr)[e_0,[e_0,e_1]],
\\[2mm]
\mathrm{ev}_{\varepsilon_2}(\xi)&=
(\log(\alpha)x_\alpha(\xi)+\log(\beta)x_\beta(\xi))e_0+(\log(\alpha)y_\alpha(\xi)+\log(\beta)y_\beta(\xi))e_1\\ &+
\left(-c_{\chi_1}z_1(\xi)-c_{\chi_2}z_2(\xi)+c_{\tau_\alpha\tau_\beta}(x_\alpha(\xi)y_\beta(\xi)-x_\beta(\xi)y_\alpha(\xi))\right)[e_0,e_1]\\ &+
\Bigl(c_{\kappa_1}k_1(\xi)+c_{\kappa_2}k_2(\xi)-c_{\tau_\alpha\chi_1}x_\alpha(\xi)z_1(\xi)-c_{\tau_\alpha\chi_2}x_\alpha(\xi)z_2(\xi)\\
&\quad -c_{\tau_\beta\chi_1}x_\beta(\xi)z_1(\xi)-c_{\tau_\beta\chi_2}x_\beta(\xi)z_2(\xi)\Bigr)[e_0,[e_0,e_1]],\\[2mm]
\mathrm{ev}_{\varepsilon_3}(\xi)&=(\log(\alpha)x_\alpha(\xi)-\log(\beta)x_\beta(\xi))e_0+(\log(\alpha)y_\alpha(\xi)-\log(\beta)y_\beta(\xi))e_1\\ &+
\left(c_{\chi_1}z_1(\xi)-c_{\chi_2}z_2(\xi)-c_{\tau_\alpha\tau_\beta}(x_\alpha(\xi)y_\beta(\xi)-x_\beta(\xi)y_\alpha(\xi))\right)[e_0,e_1]\\ &+
\Bigl(c_{\kappa_1}k_1(\xi)-c_{\kappa_2}k_2(\xi)+c_{\tau_\alpha\chi_1}x_\alpha(\xi)z_1(\xi)-c_{\tau_\alpha\chi_2}x_\alpha(\xi)z_2(\xi)\\
&\quad -c_{\tau_\beta\chi_1}x_\beta(\xi)z_1(\xi)+c_{\tau_\beta\chi_2}x_\beta(\xi)z_2(\xi)\Bigr)[e_0,[e_0,e_1]],\\[2mm]
\mathrm{ev}_{\varepsilon_4}(\xi)&=(\log(\alpha)x_\alpha(\xi)-\log(\beta)x_\beta(\xi))e_0+(\log(\alpha)y_\alpha(\xi)-\log(\beta)y_\beta(\xi))e_1\\
&+
\left(-c_{\chi_1}z_1(\xi)+c_{\chi_2}z_2(\xi)-c_{\tau_\alpha\tau_\beta}(x_\alpha(\xi)y_\beta(\xi)-x_\beta(\xi)y_\alpha(\xi))\right)[e_0,e_1]\\ &+
\Bigl(c_{\kappa_1}k_1(\xi)-c_{\kappa_2}k_2(\xi)-c_{\tau_\alpha\chi_1}x_\alpha(\xi)z_1(\xi)+c_{\tau_\alpha\chi_2}x_\alpha(\xi)z_2(\xi)\\
&\quad +c_{\tau_\beta\chi_1}x_\beta(\xi)z_1(\xi)-c_{\tau_\beta\chi_2}x_\beta(\xi)z_2(\xi)\Bigr)[e_0,[e_0,e_1]].
\end{align*}
Eliminating the $x_i(\xi), y_i(\xi), z_i(\xi), k_i(\xi)$, we get all the equations in depth $3$.

\begin{theorem}
\label{thm:depth3-equations}
Let $K=\bQ(\zeta_8)$ and $S=\left\{(1-\zeta_8)\right\}$. If $c_{\chi_i},c_{\kappa_i}\ne 0$ for $i=1,2$, then the depth-3 polylogarithmic Chabauty--Kim locus $X(\cO_K\otimes \bZ_p)_{S,\PL,3}$ is the solution set of the following equations in addition to \eqref{eq:keyeq} (with the coordinates $(X_i,Y_i,Z_i,W_i)=(\log(z_i),L_1(z_i),L_2(z_i),L_3(z_i))$ for $i=1,2,3,4$): 
\begin{equation}
\label{eq:depth3}
\begin{cases}
W_1-W_2+W_3-W_4=c_{3,1}(X_1+X_3)(Z_1+Z_3)+c_{3,2}(X_1-X_3)(Z_1+Z_4)\\
W_1-W_2-W_3+W_4=c_{3,3}(X_1-X_3)(Z_1+Z_3)+c_{3,4}(X_1+X_3)(Z_1+Z_4)
\end{cases}
\end{equation}
where
$$c_{3,1}=\frac{c_{\tau_\alpha\chi_1}}{c_{\tau_\alpha}c_{\chi_1}},\quad 
c_{3,2}=\frac{c_{\tau_\beta\chi_2}}{c_{\tau_\beta}c_{\chi_2}},\quad 
c_{3,3}=\frac{c_{\tau_\beta\chi_1}}{c_{\tau_\beta}c_{\chi_1}},\quad 
c_{3,4}=\frac{c_{\tau_\alpha\chi_2}}{c_{\tau_\alpha}c_{\chi_2}}.$$
\end{theorem}

\begin{lemma}
    For $i=1,2$, if $\Li_3(\zeta_8)+(-1)^{i+1} \Li_3(-\zeta_8)\ne 0$, then $c_{\kappa_i}\ne 0$. This is the case for all $p<200$ which split completely in $\bQ(\zeta_8)$. In particular, the nonvanishing assumption in \Cref{thm:depth3-equations} is satisfied for those~$p$.
\end{lemma}

\begin{proof}
By the evaluation maps above, $c_{\kappa_1}=0$ implies $W_1+W_2+W_3+W_4=4c_{\kappa_1}k_1(\xi) = 0$, and $c_{\kappa_2}=0$ implies $W_1+W_2-W_3-W_4=4c_{\kappa_2}k_2(\xi)=0$. Plugging in $z=\zeta_8\in X(\cO_{K,S})$ yields $\Li_3(\zeta_8)\pm \Li_3(-\zeta_8)=0$. 
\end{proof}

Using known points in $X(\cO_{K,S})$, the coefficients are computed as follows (writing $\zeta\coloneqq \zeta_8$):
\begin{align*}
    c_{3,1}&=\frac{L_3(\zeta^3+\zeta^2)-L_3(-\zeta-\zeta^2)+L_3(-\zeta^3+\zeta^2)-L_3(\zeta-\zeta^2)}{(\log(\zeta^3+\zeta^2)+\log(-\zeta^3+\zeta^2))(L_2(\zeta^3+\zeta^2)+L_2(-\zeta^3+\zeta^2))},\\[2mm]
    c_{3,2}&=\frac{L_3(\frac{1+\zeta}{2})-L_3(\frac{1-\zeta^3}{2})+L_3(\frac{1-\zeta}{2})-L_3(\frac{1+\zeta^3}{2})}{(\log(\frac{1+\zeta}{2})-\log(\frac{1-\zeta}{2}))(L_2(\frac{1+\zeta}{2})+L_2(\frac{1+\zeta^3}{2}))},\\[2mm]
    c_{3,3}&=\frac{L_3(\zeta^3+\zeta^2)-L_3(-\zeta-\zeta^2)-L_3(-\zeta^3+\zeta^2)+L_3(\zeta-\zeta^2)}{(\log(\zeta^3+\zeta^2)-\log(-\zeta^3+\zeta^2))(L_2(\zeta^3+\zeta^2)+L_2(-\zeta^3+\zeta^2))},\\[2mm]
    c_{3,4}&=\frac{L_3(\frac{1+\zeta}{2})-L_3(\frac{1-\zeta^3}{2})-L_3(\frac{1-\zeta}{2})+L_3(\frac{1+\zeta^3}{2})}{(\log(\frac{1+\zeta}{2})+\log(\frac{1-\zeta}{2}))(L_2(\frac{1+\zeta}{2})+L_2(\frac{1+\zeta^3}{2}))}.
\end{align*}

For the following discussion, we first introduce the Coleman--Sinnott functional equation.
\begin{lemma}[Coleman--Sinnott]
\label{lem:Coleman-Sinnott}
For $n\geq 1$ and $z\in X(\bZ_p)$,
\begin{equation}
\label{eq:Coleman-Sinnott}
    \mathrm{Li}_n(z)+(-1)^n\mathrm{Li}_n(z^{-1})=-\frac{1}{n!}\log(z)^n.
\end{equation}
\end{lemma}
Noting that $L_n(z)=\mathrm{Li}_n(z)$ when $z$ is a root of unity, the Coleman--Sinnott functional equation \eqref{eq:Coleman-Sinnott} shows that for all quadruples of the form $(\zeta,\zeta^{-1}, \eta,\eta^{-1})$ where $\zeta, \eta \neq 1$ are roots of unity in~$\bZ_p$, we have
    \begin{gather*}
        \text{all } X_i = 0, \quad Y_1 = Y_2,\quad Y_3 = Y_4, \quad
        Z_1 + Z_2= 0,\quad Z_3 + Z_4 = 0, \quad W_1=W_2,\quad W_3=W_4,
    \end{gather*}
    which implies that all equations of \eqref{eq:keyeq} and \eqref{eq:depth3} hold. This is not a coincidence. In fact, these points persist in infinite depth:

\begin{theorem}
\label{thm:rootofunityalwayssol}
    Assume the $p$-adic period conjecture. For $K=\bQ(\zeta_8)$, $S=\left\{(1-\zeta_8)\right\}$, the refined locus $X(\cO_K \otimes \bZ_p)_{S,\PL,\infty}^{(1)}$ for $\Sigma=1$ contains all quadruples of the form $(\zeta,\zeta^{-1}, \eta,\eta^{-1})$ where $\zeta, \eta \neq 1$ are roots of unity in~$\bZ_p$.
\end{theorem}
\begin{proof}
To show $(\zeta,\zeta^{-1}, \eta,\eta^{-1})\in X(\cO_K \otimes \bZ_p)_{S,\PL,\infty}^{(1)}$, it suffices to show that there exists $\xi\in \Sel_{S,\PL,\infty}(X)(\bQ_p)$ such that $\loc_p(\xi)=j_p((\zeta,\zeta^{-1}, \eta,\eta^{-1}))$ and $x_\alpha(\xi)=x_{\beta}(\xi)=0$.
Using the $p$-adic period conjecture, choose $\sigma_{n,i}\in (L_S^{\MT})_{-n} \otimes \bQ_p$ such that $c_{\sigma_{n,i}}=1$ for all $n\geq 2$ and $i=1,2$. (By \Cref{period-conjecture-implies-nonvanishing}, $c_{\sigma_{n,i}}\ne 0$, and we can rescale it to $1$.) The $\fp_k$-adic period element ($k=1,\dots,4$) can be written as
$$\varepsilon_k = \log(\alpha)\sigma_k^*\tau_\alpha + \log(\beta)\sigma_k^*\tau_\beta + \sum_{n=2}^\infty (\sigma_k^* \sigma_{n,1} + \sigma_k^* \sigma_{n,2})\bmod{\text{commutators}}$$
where $\sigma_k\in \Gal(K/\bQ)$ such that $\fp_k=\sigma_{k}^*\fp_1$.

Note that $\xi$ maps commutators to $0$ if $x_\alpha(\xi)=x_{\beta}(\xi)=0$. Under this condition, $\ev_{\varepsilon_k}(\xi)$ is given by
\begin{equation*}
\ev_{\varepsilon_k}(\xi)=\scriptsize{
\begin{cases}
(\log(\alpha)y_\alpha(\xi)+\log(\beta)y_\beta(\xi))e_1+\sum_{n=2}^\infty (z_{n,1}(\xi)+z_{n,2}(\xi))\ad(e_0)^{n-1}e_1, & k=1; \\
(\log(\alpha)y_\alpha(\xi)+\log(\beta)y_\beta(\xi))e_1+\sum_{n=2}^\infty ((-1)^{n+1}z_{n,1}(\xi)+(-1)^{n+1}z_{n,2}(\xi))\ad(e_0)^{n-1}e_1, & k=2; \\
(\log(\alpha)y_\alpha(\xi)-\log(\beta)y_\beta(\xi))e_1+\sum_{n=2}^\infty (z_{n,1}(\xi)-z_{n,2}(\xi))\ad(e_0)^{n-1}e_1, & k=3; \\
(\log(\alpha)y_\alpha(\xi)-\log(\beta)y_\beta(\xi))e_1+\sum_{n=2}^\infty ((-1)^{n+1}z_{n,1}(\xi)+(-1)^n z_{n,2}(\xi))\ad(e_0)^{n-1}e_1, & k=4.
\end{cases}}
\end{equation*}

On the other hand, write $(z_1,z_2,z_3,z_4)=(\zeta,\zeta^{-1},\eta,\eta^{-1})$ and notice that $\log(z_k)=0$ for all $k=1,\ldots,4$, then we need
$$\ev_{\varepsilon_k}(\xi)=\Li_1(z_k)e_1+\sum_{n=2}^\infty \Li_n(z_k)\ad(e_0)^{n-1}e_1.$$
Comparing the coefficients gives equations for $y_{\alpha}(\xi), y_{\beta}(\xi)$ and the $z_{n,i}(\xi)$. Note that we only need the equations for $k=1,3$ by Coleman--Sinnott (\Cref{lem:Coleman-Sinnott}), which implies $\Li_n(z^{-1})=(-1)^{n+1}\Li_n(z)$ for roots of unity $z\ne 1$ (so the equations for $k=2$ are implied by those for $k=1$, and those for $k=4$ are implied by those for $k=3$). Thus, the equations are given by
$$\begin{cases}\log(\alpha)y_\alpha(\xi)+\log(\beta)y_\beta(\xi)=\Li_1(\zeta), & \\ 
\log(\alpha)y_\alpha(\xi)-\log(\beta)y_\beta(\xi)=\Li_1(\eta), & \\
z_{n,1}(\xi)+z_{n,2}(\xi) = \Li_n(\zeta), & \text{ for }n\geq 2,\\
z_{n,1}(\xi)-z_{n,2}(\xi) = \Li_n(\eta), & \text{ for }n\geq 2.
\end{cases}$$

We solve 
\begin{alignat*}{3}
&y_\alpha(\xi)=\frac{\Li_1(\zeta)+\Li_1(\eta)}{2\log(\alpha)}, \qquad && y_\beta(\xi)=\frac{\Li_1(\zeta)-\Li_1(\eta)}{2\log(\beta)},\\
&z_{n,1}(\xi)=\frac{\Li_n(\zeta)+\Li_n(\eta)}{2},\qquad && z_{n,2}(\xi)=\frac{\Li_n(\zeta)-\Li_n(\eta)}{2}. \qedhere
\end{alignat*}
\end{proof}

By \Cref{thm:rootofunityalwayssol}, the depth-3 locus contains at least $X(\cO_{K,S})$ and $\mathcal{R}_p$ where
\begin{equation}
\label{eq:depth3solgenform}
  \mathcal{R}_p \coloneqq 
\bigcup_{\zeta,\eta\in\mu_{p-1}(\bZ_p)\setminus\{1\}}
\left\{
(\zeta,\zeta^{-1},\eta,\eta^{-1})
\right\},
\end{equation}
and also contains $X(\bZ[\sqrt{2}]\otimes\bZ_p)_{(\sqrt{2}),\PL,3}$ by functoriality, which is the same as $X(\bZ[\sqrt{2}]\otimes\bZ_p)_{(\sqrt{2}),\PL,2}$ (see \cite[§9.2]{LL:PolylogNF}).

\begin{theorem}
\label{thm:depth3-locus-description}
Assume the $p$-adic period conjecture. For $K=\bQ(\zeta_8)$, $S=\left\{(1-\zeta_8)\right\}$ and primes $p$ splitting completely in $K$,
\begin{align*}
    X(\cO_{K,S})\cup \mathcal{R}_p \cup X(\bZ[\sqrt{2}]\otimes\bZ_p)_{(\sqrt{2}),\PL,2}&\subseteq X(\cO_K\otimes\bZ_p)_{S,\mathrm{PL},3},\\
    X(\cO_{K,S})\cup \mathcal{R}_p\cup X(\bZ[\sqrt{2}]\otimes\bZ_p)_{(\sqrt{2}),\PL,2}^{\min}&\subseteq X(\cO_K\otimes\bZ_p)_{S,\mathrm{PL},3}^{\min}.
\end{align*}
All of the above containments are equalities unconditionally for all primes $p<200$ which split completely in $K$.
\end{theorem}

Unlike the equations for depth $2$ (\eqref{eq:keyeq}), the equations for depth $3$ (\eqref{eq:keyeq} + \eqref{eq:depth3}) are not $S_3$-stable, i.e., the ideal generated by them is not stable under the $S_3$-action. This is because the definition of the polylogarithmic quotient breaks the $S_3$-symmetry. (Recall that the polylogarithmic quotient is based on the inclusion $X\hookrightarrow \bG_m$, which distinguishes $1$ as the puncture being filled in.) We can implement the technique of \textit{$S_3$-symmetrisation} \cite[§5.2]{CDC:polylog1} to define an intermediate set
\[ X(\cO_K \otimes \bZ_p)_{S,\PL,N} \supseteq X(\cO_K \otimes \bZ_p)_{S,\PL,N}^{S_3} \supseteq X(\cO_K \otimes \bZ_p)_{S,N} \]
strengthening the polylogarithmic Chabauty--Kim method. The $S_3$-symmetrised locus $X(\cO_K \otimes \bZ_p)_{S,\PL,N}^{S_3}$ is defined as
$$X(\cO_K\otimes \bZ_p)_{S,\PL,N}^{S_3}\coloneqq \bigcap_{\sigma\in S_3} \sigma(X(\cO_K\otimes \bZ_p)_{S,\PL,N})$$
where $S_3$ acts on $X(\cO_K\otimes \bZ_p)_{S,\PL,N}$ componentwise.

Here are the statistics for the depth-3 polylogarithmic Chabauty--Kim locus in depth $3$ with $p=17$, including $S_3$-symmetrisation:
\smallskip

\resizebox{\textwidth}{!}{
\begin{tabular}{ |c|c|c|c|c|c|c|c|c|c|c|c|c| } 
 \hline
 \textbf{type} & \multicolumn{2}{c|}{\textbf{Standard}} & \multicolumn{2}{c|}{\textbf{$S_3$-symm}} & \multicolumn{2}{c|}{\textbf{Refined ($\Sigma=1$)}} & \multicolumn{2}{c|}{\textbf{Refined ($\Sigma=0$)}} & \multicolumn{2}{c|}{\textbf{Refined Total}} & \multicolumn{2}{c|}{\textbf{Refined $S_3$-symm}} \\ \hline
  & Sols & Exc. & Sols & Exc. & Sols & Exc. & Sols & Exc. & Sols & Exc. & Sols & Exc. \\ \hline
 \textbf{off-off} & $232$ & $190$ & $42$  & $0$   & $204$ & $190$ & $14$ & $0$  & $232$ & $190$ & $42$  & $0$  \\ \hline
 \textbf{on-off}  & $14$  & $14$  & $0$   & $0$   & $14$  & $14$  & $0$  & $0$  & $14$  & $14$  & $0$   & $0$  \\ \hline
 \textbf{off-on}  & $14$  & $14$  & $0$   & $0$   & $14$  & $14$  & $0$  & $0$  & $14$  & $14$  & $0$   & $0$  \\ \hline
 \textbf{on-on}   & $183$ & $150$ & $183$ & $150$ & $21$  & $10$  & $21$ & $10$ & $63$  & $30$  & $63$  & $30$ \\ \hline
 \textbf{TOTAL}   & $443$ & $368$ & $225$ & $150$ & $253$ & $228$ & $35$ & $10$ & $323$ & $248$ & $105$ & $30$ \\ \hline
\end{tabular}}

\subsection{Depth 4}
The table below shows the basis elements in $(L^{\mathrm{MT}}_S)_{-4}$ modulo the Goncharov ideal with their Galois action.
\smallskip

\resizebox{\textwidth}{!}{%
\begin{minipage}{1.2\textwidth}
\centering
\begin{tabular}{ |c|c|c|c|c|c|c|c|c|c| } 
 \hline
 $\mathrm{id}$ & $[\tau_\alpha,[\tau_\alpha,\chi_1]]$ & $[\tau_\alpha,[\tau_\alpha,\chi_2]]$ & $[\tau_\alpha,\kappa_1]$ & $[\tau_\alpha,\kappa_2]$ & $[\tau_\beta,[\tau_\alpha,\chi_1]]$ & $[\tau_\beta,[\tau_\alpha,\chi_2]]$ & $[\tau_\beta,[\tau_\beta,\chi_1]]$ & $[\tau_\beta,[\tau_\beta,\chi_2]]$ \\ \hline
 $\sigma_+$ & $-1$ & $-1$ & $1$ & $1$ & $-1$ & $-1$ & $-1$ & $-1$ \\ \hline
 $\sigma_i$ & $1$ & $-1$ & $1$ & $-1$ & $-1$ & $1$ & $1$ & $-1$ \\ \hline
\end{tabular} \\[3mm]
\begin{tabular}{ |c|c|c|c|c|c|c|c|c| } 
 \hline
 $\mathrm{id}$ & $[\tau_\beta,\kappa_1]$ & $[\tau_\beta,\kappa_2]$ & $[\tau_\alpha,[\tau_\alpha,[\tau_\alpha,\tau_\beta]]]$ & $[\tau_\beta,[\tau_\alpha,[\tau_\alpha,\tau_\beta]]]$ & $[\tau_\beta,[\tau_\beta,[\tau_\alpha,\tau_\beta]]]$ & $\sigma_{4,1}$ & $\sigma_{4,2}$ \\ \hline
 $\sigma_+$ &  $1$ & $1$ & $1$ & $1$ & $1$ & $-1$ & $-1$ \\ \hline
 $\sigma_i$ & $-1$ & $1$ & $-1$ & $1$ & $-1$ & $1$ & $-1$ \\ \hline
\end{tabular}
\end{minipage}}
\smallskip

Let $\mu_1,\mu_2\in (L_S^{\MT})_{-4} \otimes \bQ_p$ with $(\sigma_+,\sigma_i)\mu_1=(-1,1)\mu_1$ and $(\sigma_+,\sigma_i)\mu_2=(-1,-1)\mu_2$ which absorb the terms in the expansion of $\eps_1$ with the same Galois action, as we did for $\kappa_1,\kappa_2$. We write the new global coordinates in addition to \eqref{eq:depth2-parametrisation}, \eqref{eq:depth3-parametrisation} for depth $4$ as
\begin{equation}
\label{eq:depth4-parametrisation}
   \xi(\mu_1)=u_1(\xi)[e_0,[e_0,[e_0,e_1]]],\qquad \xi(\mu_2)=u_2(\xi)[e_0,[e_0,[e_0,e_1]]]. 
\end{equation}

Write the $\fp_1$-adic period point modulo degree $<-4$ and modulo the Goncharov ideal as
\begin{align*}
    \varepsilon_1&= \log(\alpha)\tau_\alpha+\log(\beta)\tau_\beta+c_{\chi_1}\chi_1+c_{\chi_2}\chi_2+c_{\tau_\alpha\tau_\beta}[\tau_\alpha,\tau_\beta]+c_{\kappa_1}\kappa_{1}+c_{\kappa_2}\kappa_{2}\\
    &+c_{\tau_\alpha\chi_1}[\tau_\alpha,\chi_1]+c_{\tau_\alpha\chi_2}[\tau_\alpha,\chi_2]+c_{\tau_\beta\chi_1}[\tau_\beta,\chi_1]+c_{\tau_\beta\chi_2}[\tau_\beta,\chi_2] \\
        &+ c_{\mu_1}\mu_1+c_{\mu_2}\mu_2 +c_{\tau_\alpha\kappa_1}[\tau_\alpha,\kappa_1]+c_{\tau_\alpha\kappa_2}[\tau_\alpha,\kappa_2] +c_{\tau_\beta\kappa_1}[\tau_\beta,\kappa_1] +c_{\tau_\beta\kappa_2}[\tau_\beta,\kappa_2] \\
        &+ c_{\tau_\alpha\tau_\alpha\tau_\alpha\tau_\beta}[\tau_\alpha,[\tau_\alpha,[\tau_\alpha,\tau_\beta]]]+c_{\tau_\beta\tau_\beta\tau_\alpha\tau_\beta}[\tau_\beta,[\tau_\beta,[\tau_\alpha,\tau_\beta]]]
    + c_{\tau_\beta\tau_\alpha\tau_\alpha\tau_\beta}[\tau_\beta,[\tau_\alpha,[\tau_\alpha,\tau_\beta]]]
\end{align*}
From this we obtain $\varepsilon_2,\varepsilon_3,\varepsilon_4$ by applying the Galois action. We get formulas for the evaluation maps $\ev_{\eps_i}$ as consequence, and finally obtain equations for the Chabauty--Kim locus by eliminating the variables. We omit the computation details and directly present the resulting depth-4 equations:

\begin{theorem}
\label{thm:depth4-equations}
Let $K=\bQ(\zeta_8)$ and $S=\left\{(1-\zeta_8)\right\}$. If $c_{\chi_i},c_{\kappa_i},c_{\mu_i}\ne 0$ for $i=1,2$ and $c_{\tau_\alpha\tau_\beta}\ne 0$, the depth-4 polylogarithmic Chabauty--Kim locus $X(\cO_K\otimes \bZ_p)_{S,\PL,4}$ is the solution set of the following equations in addition to \eqref{eq:keyeq} and \eqref{eq:depth3} (with the coordinates $(X_i,Y_i,Z_i,W_i,V_i)=(\log(z_i),L_1(z_i),L_2(z_i),L_3(z_i),L_4(z_i))$ for $i=1,2,3,4$):
\begin{equation}
\label{eq:depth4eq}
\left\{
\begin{aligned}
& V_{1}+V_{2}+V_{3}+V_{4} = c_{4,1} (X_{1}+X_{3})(W_{1}+W_{2}+W_{3}+W_{4})  \\
& \quad + c_{4,2} (X_{1}-X_{3})(W_{1}+W_{2}-W_{3}-W_{4}) + c_{4,3} (X_{1}-X_{3})(X_{1}+X_{3})(Z_{1}+Z_{2})  \\[4pt]
& V_{1}+V_{2}-V_{3}-V_{4} = c_{4,4} (X_{1}-X_{3})(W_{1}+W_{2}+W_{3}+W_{4})  \\
& \quad + c_{4,5} (X_{1}+X_{3})(W_{1}+W_{2}-W_{3}-W_{4}) \\ &\quad+ c_{4,6} (X_{1}-X_{3})^{2}(Z_{1}+Z_{2}) + c_{4,7} (X_{1}+X_{3})^{2}(Z_{1}+Z_{2}) 
\end{aligned}
\right.
\end{equation}
where
\begin{gather}
c_{4,1} = \frac{c_{\tau_\alpha \kappa_1}}{2c_{\tau_\alpha} c_{\kappa_1} }, \quad
c_{4,2} = \frac{c_{\tau_\beta \kappa_2}}{2c_{\tau_\beta}c_{\kappa_2}}, \quad
c_{4,3} = \frac{c_{\tau_\beta\tau_\alpha\tau_\alpha\tau_\beta}}{2 c_{\tau_\alpha\tau_\beta} c_{\tau_\alpha} c_{\tau_\beta}}, \\
c_{4,4} = \frac{c_{\tau_\beta \kappa_1}}{2 c_{\tau_\beta} c_{\kappa_1}}, \quad
c_{4,5} = \frac{c_{\tau_\alpha \kappa_2}}{2 c_{\tau_\alpha} c_{\kappa_2}}, \quad
c_{4,6} = \frac{c_{\tau_\beta\tau_\beta\tau_\alpha\tau_\beta}}{2 c_{\tau_\alpha\tau_\beta} c_{\tau_\beta}^2},\quad 
c_{4,7} = \frac{c_{\tau_\alpha\tau_\alpha\tau_\alpha\tau_\beta}}{2 c_{\tau_\alpha\tau_\beta} c_{\tau_\alpha}^2}.
\end{gather}
\end{theorem}

\begin{lemma}
    For $i=1,2$, if $\Li_4(\zeta_8)+(-1)^{i+1} \Li_4(-\zeta_8)\ne 0$, then $c_{\mu_i}\ne 0$. The conditions $\Li_4(\zeta_8)\pm \Li_4(-\zeta_8)\ne 0$ and $c_{\tau_\alpha\tau_\beta}=L_2(\beta)\ne 0$ (see \eqref{eq:ctaualphataubeta}) are satisfied for all $p<200$ which split completely in $K=\bQ(\zeta_8)$. In particular, the nonvanishing assumption in \Cref{thm:depth4-equations} is satisfied for those~$p$.
\end{lemma}
\begin{proof}
By the evaluation maps, $c_{\mu_1}=0$ implies $V_1-V_2+V_3-V_4=4c_{\mu_1}u_1(\xi)=0$, and $c_{\mu_2}=0$ implies $V_1-V_2-V_3+V_4=4c_{\mu_2}u_2(\xi)=0$. Plugging in $z=\zeta_8\in X(\cO_{K,S})$ gives $\Li_4(\zeta_8)\pm \Li_4(-\zeta_8)=0$. 
\end{proof}

The coefficients $c_{4,1},\ldots,c_{4,7}$ can be determined by solving the linear equations obtained by plugging points from $X(\cO_{K,S})$ into \eqref{eq:depth4eq}. Once they are determined, we can compute the depth-4 polylogarithmic Chabauty--Kim locus. The results for various~$p$ are shown in the following table.
\begin{table}[ht]
\centering
\resizebox{\textwidth}{!}{
\begin{tabular}{ |c|c|c|c|c|c|c|c|c| } 
 \hline
 $\# X(\cO_K\otimes \bZ_p)_{S,\PL,N}$ & $p=17$  & $p=41$ & $p=73$ & $p=89$ & $p=97$ & $p=113$ & $p=137$ & $p=193$
  \\  \hline
 $N=2$ & $903$ & $6255$ & $20509$ & $30111$ & $36001$ & $48975$ & $73503$ & $146947$ \\    
 $N=3$ & $443$ & $3203$ & $10161$ & $14963$ & $17589$ & $24323$ & $37043$ & $73263$ \\
 $N=4$ & $293$ & $1589$ & $5109$ & $7637$ &  $9093$ & $12389$ & $18293$ & $36549$ \\
 $N=4$ + $S_3$-symmetrisation & $75$ & $75$ & $79$ & $75$ & $79$ & $75$ & $75$ & $79$ \\
 \hline
\end{tabular}}
\caption{The cardinality of $X(\cO_K\otimes \bZ_p)_{S,\mathrm{PL},N}$ for $N=2,3,4$ (including $S_3$-symmetrisation for $N=4$), for splitting primes $p<200$}
\label{tab:overallstatistics}
\end{table}

We see that $\# X(\cO_K\otimes \bZ_p)_{S,\PL,4}=(p-2)^2+68$ for all primes in the table. This fact implies that the depth $4$ locus only contains $X(\cO_{K,S})$ plus the roots-of-unity solutions provided by \Cref{thm:rootofunityalwayssol}, which is the best we can hope for.
\begin{theorem}
\label{thm:depth4-locus-description}
    Assume the $p$-adic period conjecture. For $K=\bQ(\zeta_8)$, $S=\left\{(1-\zeta_8)\right\}$ and primes $p$ splitting completely in $K$ (see the definition of $\mathcal{R}_p$ in \eqref{eq:depth3solgenform}),
    \begin{equation}
    \label{eq:XOKSPLdepth4}
      X(\cO_{K,S})\cup \mathcal{R}_p\subseteq X(\cO_K\otimes \bZ_p)_{S,\PL,4}^{\min} \subseteq  X(\cO_K\otimes \bZ_p)_{S,\PL,4}.
    \end{equation}
    Both containments are equalities unconditionally for all primes $p<200$ which split completely in $K$.
\end{theorem}

Because $\mathcal{R}_p$ persists in the refined locus even in infinite depth (by \Cref{thm:rootofunityalwayssol}, assuming the $p$-adic period conjecture), we do not expect Kim's Conjecture to hold for the polylogarithmic quotient. We may ask whether the polylogarithmic Chabauty--Kim method in combination with $S_3$-symmetrisation (see §\ref{subsection:depth3}) is enough to prove Kim's Conjecture (for the full depth-$N$ quotient), i.e., whether $X(\cO_K\otimes \bZ_p)_{S,\PL,N}^{S_3}=X(\cO_{K,S})$ for $N \gg 0$. Since $\# X(\cO_{K,S})=75$, we read off the last row of the table that Kim's Conjecture holds in depth~4 for $p=17,41,89,113,137$, which proves \Cref{mainthm:s3-sym-kim}. The remaining cases $p=73, 97, 193$ are exactly those $p<200$ where $\zeta_6\in \bZ_p$, and the $79-75=4$ exceptional solutions are $(\zeta_6^{\pm 1},\zeta_6^{\mp 1},\zeta_6^{\pm 1},\zeta_6^{\mp 1})$ and $(\zeta_6^{\pm 1},\zeta_6^{\mp 1},\zeta_6^{\mp 1},\zeta_6^{\pm 1})$.

\section*{Data, Materials, and Software Availability}
The Sage code and data used for obtaining computational results is available at \url{https://github.com/martinluedtke/PolylogNF}.

\printbibliography

\end{document}